\documentclass[11pt,a4paper]{amsart}

\usepackage{graphicx}
\usepackage{enumerate}
\usepackage{amsmath, amssymb, amsthm, amscd}
\usepackage{pst-all}
\usepackage{url}
\usepackage{caption}
\usepackage{xcolor}

\theoremstyle{plain}
\newtheorem{theorem}{Theorem}[section]

\newtheorem{proposition}[theorem]{Proposition}

\theoremstyle{definition}
\newtheorem{definition}[theorem]{Definition}

\theoremstyle{remark}
\newtheorem{remark}[theorem]{Remark}

\numberwithin{equation}{section}

\def\bold#1{\mbox{\boldmath $#1$}}
\newcommand{\uu}[1]{\bold{#1}}
\newcommand{\mbu}{\uu{u}}
\newcommand{\mbq}{\uu{q}}
\newcommand{\ba}{\underline{a}} 
\newcommand{\br}{\underline{\rho}}
\newcommand{\bu}{\underline{\mbu}}
\newcommand{\baru}{\underline{u}}
\newcommand{\const}{\mathrm{const.}}
\newcommand{\abs}[1]{\lvert#1\rvert}
\newcommand{\D}{\partial}

\newcommand{\Dlt}{\Delta t}

\newcommand{\Dt}{\partial_t}
\newcommand{\dvg}{\nabla \cdot }

\newcommand{\grd}{\nabla}

\newcommand{\ld}{\lambda}

\newcommand{\mbb}{\mathbb}

\newcommand{\mcal}{\mathcal}

\newcommand{\norm}[1]{\lVert#1\rVert}
\newcommand{\Norm}[1]{{\left\vert\kern-0.25ex\left\vert\kern-0.25ex\left\vert #1 
    \right\vert\kern-0.25ex\right\vert\kern-0.25ex\right\vert}}

\newcommand{\rf}{\mathrm{ref}}
\newcommand{\CWENO}{\mathrm{CWENO}}

\newcommand{\solspc}{L^2(\mbb{T}^d)^{1+d}}

\newcommand{\veps}{\varepsilon}

\def\bold#1{\mbox{\boldmath $#1$}}

\makeindex             

\begin{document}

\title[Low Mach Number Accurate IMEX Finite Volume Schemes]{Asymptotic
  Preserving Low Mach Number Accurate IMEX Finite Volume Schemes
  for the Isentropic Euler Equations}    

\author[Arun]{K. R. Arun}
\address{School of Mathematics, Indian Institute of Science Education
  and Research Thiruvananthapuram, Thiruvananthapuram - 695551,
  India.} 
\email{arun@iisertvm.ac.in}

\author[Samantaray]{S. Samantaray}
\address{School of Mathematics, Indian Institute of Science Education
  and Research Thiruvananthapuram, Thiruvananthapuram - 695551,
  India.} 
\email{sauravsam13@iisertvm.ac.in}

\date{\today}

\subjclass[2010]{Primary 35L45, 35L60, 35L65, 35L67; Secondary 65M06,
  65M08}

\keywords{Compressible Euler system, Incompressible Euler system, Zero
  Mach number limit, IMEX-RK schemes, Asymptotic preserving,
  Asymptotic accuracy, Finite volume method}

\begin{abstract}
 In this paper, the design and analysis of a class of second order
 accurate IMEX finite volume schemes for the compressible Euler
 equations in the zero Mach number limit is presented. In order to
 account for the fast and slow waves, the nonlinear fluxes in the
 Euler equations are split into stiff and non-stiff components,
 respectively. The time discretisation is performed by an IMEX
 Runge-Kutta method, therein the stiff terms are treated implicitly
 and the non-stiff terms explicitly. In the space discretisation, a
 Rusanov-type central flux is used for the non-stiff part, and simple
 central differencing for the stiff part. Both the time semi-discrete
 and space-time fully-discrete schemes are shown to be asymptotic
 preserving. The numerical experiments confirm that the schemes
 achieve uniform second order convergence with respect to the Mach
 number. A notion of accuracy at low Mach numbers, termed as the
 asymptotic accuracy, is introduced in terms of the invariance of a
 well-prepared space of constant densities and divergence-free
 velocities. The asymptotic accuracy is concerned with the closeness
 of the compressible solution with that of its incompressible
 counterpart in a low Mach number regime. It is shown theoretically as
 well as numerically that the proposed schemes are asymptotically
 accurate.        
\end{abstract}

\maketitle

\section{Introduction}
\label{sec:sec1}

Many physical phenomena in hydrodynamics and magnetohydrodynamics are  
governed by the compressible Euler equations which represent the
fundamental conservation principles of mass, momentum and
energy. In many problems from meteorology, geophysics, combustion
etc., one often encounters a situation where a characteristic fluid
velocity is much lesser than a corresponding sound velocity in the
medium. In such cases, the ratio of these two velocities typically gives
rise to a singular perturbation parameter, such as the Mach number or 
the Froude number. When the motion of the fluid is humble in
comparison to the motion of sound waves in the medium, i.e.\ when
the Mach number tends to zero, the fluid can be treated almost
incompressible. From a mathematical point of view, one can say that
solutions of the compressible Euler equations are close to those of
their incompressible counterparts as the Mach number approaches
zero. In many seminal works, e.g.\
\cite{klainerman-majda,klainermanMajda82,schochet}, a rigorous
convergence analysis of the solutions of the purely hyperbolic
compressible Euler system to those of the mixed hyperbolic-elliptic
incompressible Euler system in the limit of zero Mach number can be
found.   

It is widely known in the literature that standard explicit
Gudunov-type finite volume compressible flow solvers suffer from a lot
pathologies at low Mach numbers. The stiffness arising from stringent
CFL stability restrictions, loss of accuracy due to the creation of
spurious waves, lack of stability due to the dependence of numerical
viscosity on the Mach number, and inability to respect the 
transitional behaviour of the continuous system of governing equations
are some of the commonly observed ailments, to name but a few. The
stiffness due to CFL restrictions severely imposes the timesteps
to be as small as the order of the Mach number, resulting in a drastic
slowdown of the solver in a low Mach number regime. The inaccuracy
of the numerical solution, and inability of the numerical method to
respect the transitional behaviour of the equations are mainly due to
the loss of information in the passage from continuous to discrete
level. We refer the reader to \cite{dellacherie, guillard, klein} and
the references therein for a detailed account of the above-mentioned
anomalies.  

Design of stable and accurate numerical schemes for weakly
compressible or low Mach number flows is a challenging task due to the 
aforementioned difficulties and more. In the literature, several
approaches to develop numerical schemes which work for small 
as well as large Mach numbers, addressing one or more of the
above issues, can be found. In \cite{bijl-wesseling},
Bijl and Wesseling used the standard MAC-type finite difference method
for the incompressible Euler equations to derive a scheme which can
simulate flows at a wide range of Mach numbers. Another approach is
due to Munz et al.\ \cite{munz-etal}, in which the wellknown SIMPLE
method for incompressible flows is extended to weakly compressible
flows using the insight gained from an asymptotic analysis. Recently,
Feistauer and Ku\v{c}era \cite{feistauer-kuvcera} developed an all
Mach number flow solver in the context of high order discontinuous
Galerkin methods. We refer the interested reader also to, e.g.\
\cite{benacchio-etal, klein-etal, smolarkiewicz-etal}, for a different
approach using the so-called sound-proof models which eliminate sound
waves completely and provide good approximations for low speed
atmospheric flows.   

The convergence of solutions of the compressible Euler equations to
those of the incompressible Euler equations in the limit of zero Mach
number, and the associated challenges in numerical approximation is an
active area of research. The consistency of the limit of a
compressible flow solver, if it exists, with the incompressible limit
system, and the stiffness arising from stringent stability requirements
are usually addressed in the framework of the famous asymptotic
preserving (AP) methodology. The notion of AP schemes was initially
introduced by Jin \cite{jin-ap} for kinetic transport equations; see
also \cite{Degond13, Jin12} for a comprehensive review of this
subject. The AP framework provides a systematic and robust tool for
developing numerical discretisation techniques for singular
perturbation problems. The procedure takes into account the
transitional behaviour of the governing equations of the problem, and 
its stability requirements are independent of the singular
perturbation parameter. Hence, an AP scheme for the compressible Euler
equations automatically transforms to an incompressible solver when the
Mach number goes to zero. As discussed in \cite{Degond13},
semi-implicit time-stepping techniques provide a systematic approach
to derive AP schemes; see, e.g.\ \cite{BispArunLucaNoe, cordier,
  degond_tang, NoelleArunMunz, zakerzadeh-noelle}, for some
semi-implicit AP schemes for the Euler or shallow water equations.      

The presence of singular perturbation parameters gives rise to
multiple scales in time as well as space. As mentioned above, a widely
used strategy to resolve multiple timescales, and to derive an AP
scheme is the use of semi-implicit time discretisations. The property
of being AP of a numerical approximation is primarily that of the
particular time discretisation used. Implicit-explicit Runge-Kutta 
(IMEX-RK) schemes offer a precise and robust approach to define high
order semi-implicit AP schemes for singularly perturbed ordinary
differential equations (ODEs). In the literature, one can find several
references where these schemes have been extensively used and
analysed for stiff systems of ODEs, and also time-dependent partial
differential equations. We refer the interested reader to
\cite{AscherRuuthSpiteri, pareschi-russo-conf, pareschi-russo} for the
application of IMEX schemes to stiff ODEs, to
\cite{boscarino-sinum2007} for an error analysis of IMEX schemes, and
to
\cite{boscarino-scandurra,cordier,degond_tang,dimarco_loubere_jcp_2018,
  dimarco_loubere_siamjsc_2017,haack,tang_kinrelmod_2012} for an
application of the IMEX strategy to low Mach number Euler
equations. However, the decisive step in the implementation of an
IMEX-RK method to a system of conservation laws, such as the Euler or
shallow water equations, is a splitting of the fluxes into the
so-called stiff and non-stiff components in order that the resulting
scheme is AP.      

Recently, in \cite{dellacherie,dellacherie-omnes-rieper}, the authors
have presented the results of a detailed study on the accuracy of
Godunov-type schemes at low Mach numbers. It has been shown in these
papers that the origin of inaccuracies is predominantly due to the
formation of spurious acoustic waves in the discrete solution. There
are several factors which can attributed to this, such as the number
of space dimensions, particular discretisation strategies and
numerical diffusion in the scheme used, cell geometry chosen, and so
on. Based on the work of Schochet \cite{schochet}, Dellacherie in
\cite{dellacherie} has presented the results of an extensive study on
the accuracy of Gudunov-type schemes at low Mach numbers on Cartesian
meshes. At the core of this analysis is an estimate from
\cite{schochet} which states that when the Mach number is small, a
solution to the compressible Euler equations is ``$\veps$-close'' to that
of the incompressible equations whenever the corresponding initial
data are close, where $\veps$ is a reference Mach number, cf.\ also
\cite{klainerman-majda,klainermanMajda82}. Further, an analysis of the
nature of solutions of the linear wave equation system at low Mach
numbers using Hodge decompositions shows that a sufficient condition
to ensure the above-mentioned estimate in the linear case is the
invariance of a so-called `well-prepared space' of constant densities
and divergence-free velocities. In the light of this invariance
property, Dellacherie identified the origin of inaccuracies of
standard Godunov-type schemes in two and three-dimensional geometries
using the notion of first-order modified equations. Finally, a
sufficient condition to avoid the generation of spurious waves, and to
ensure accuracy at low Mach numbers for nonlinear schemes for the
Euler equations is proposed by modifying the numerical viscosity, and
changing the discretisation of the pressure gradient term. 

In this paper, we study a class of second order accurate finite volume
schemes for the isentropic Euler equations, employing the IMEX-RK
time-stepping procedure. The flux-splitting required to design these
IMEX schemes is analogous to that used in
\cite{boscarino-scandurra,dimarco_loubere_jcp_2018,tang_kinrelmod_2012}. We
prove that both the time semi-discrete and space-time fully-discrete
schemes are asymptotically consistent with the limiting incompressible
system in the zero Mach number limit, and that their stability constraints are independent of the Mach number which
implies the AP property. We also define a notion of accuracy at low
Mach numbers, hereafter designated as the asymptotic accuracy, in
terms of the invariance of the well-prepared space as in
\cite{dellacherie}. In fact, we show that the invariance is the key
property for the numerical solution to mimic the $\veps$-closeness as
in the case of the exact solution. As result of the $\veps$-closeness,
the error remains bounded, and we obtain the accuracy in the stiff
limit $\veps=0$, justifying the name asymptotic accuracy. It is shown
theoretically as well as numerically that  both the time semi-discrete
and space-time fully-discrete schemes are asymptotically accurate. The
rest of this paper is organised as follows. In Section~\ref{sec:sec2},
we mention the zero Mach number limits of the isentropic Euler and the
linear wave equation systems, and recall the main convergence results
from \cite{dellacherie}. Section~\ref{sec:sec3} is dedicated to the 
defintion of an AP and AA scheme. In Section~\ref{sec:sec4}, we
introduce the IMEX-RK time semi-discrete schemes, and in
Section~\ref{sec:sec5} we show their AP property by giving a
sufficient condition for $L^2$-stability, and the asymptotic
accuracy. Space-time fully-discrete schemes, and their analysis
containing a similar stability condition, are presented in
Section~\ref{sec:sec6}. In Section~\ref{sec:sec7}, numerical evidences
are provided to support the theoretical claims made. Finally, we close
the paper with some concluding remarks in Section~\ref{sec:sec8}.   

\section{Zero Mach Number Limits of the Euler and the Wave Equation
  Systems}
\label{sec:sec2}

\subsection{Zero Mach Number Limit of the Euler System }
\label{sec:zero-mach-number}

The scaled, compressible Euler equations are given by 
\begin{equation} 	
  \label{eq:comp}
  \begin{aligned} 
    \Dt \rho + \dvg (\rho \mbu) &=0, \\
    \Dt (\rho\mbu) + \dvg (\rho \mbu \otimes \mbu) + \frac{\grd p}{\veps^2} &=0, 
  \end{aligned}
\end{equation}
where the independent variables are time $t>0$ and space
$\uu{x}\in\mbb{R}^d,d=1,2,3$, and the dependent variables are the density
$\rho=\rho(t,\uu{x})>0$ and velocity of the fluid
$\mbu=\mbu(t,\uu{x})\in\mbb{R}^d$. In order to close the system \eqref{eq:comp},
an isentropic equation of state $p(\rho) = \rho^\gamma$, where
$\gamma$ is a positive constant, is assumed. Here, the parameter 
$\veps$ is a reference Mach number defined by  
\begin{equation}
  \label{eq:eps_defn}
  \veps:=\frac{u_\rf}{c_\rf},
\end{equation}
with $u_\rf$ and $c_\rf$ being a reference fluid speed, and a reference
sound speed, respectively.  

In \cite{klainerman-majda}, Klainerman and Majda rigorously
investigated the limit of solutions of \eqref{eq:comp} as
$\veps\to0$. Formally, the asymptotic nature of solutions to
\eqref{eq:comp} can be studied by plugging-in the ansatz 
\begin{equation}
  \label{eq:f_ansatz}
  f(t,\uu{x})=f_{(0)}(t,\uu{x})+\veps f_{(1)}(t,\uu{x})+\veps^2 f_{(2)}(t,\uu{x})
\end{equation}
for all dependent variables. Performing a scale analysis, and using
appropriate boundary conditions, cf.\ e.g.\ \cite{meister}, yields the
multiscale representation 
\begin{equation}
  \label{eq:mScaleRho}
  p(t,\uu{x}) = p_{(0)}(t)+\veps^2p_{(2)}(t,\uu{x}), 
\end{equation}
for the pressure and the incompressible Euler system:    
\begin{equation}
  \label{eq:incomp}
  \begin{aligned}
    \Dt \mbu_{(0)}+\dvg\left(\mbu_{(0)}\otimes \mbu_{(0)}\right)+\grd p_{(2)}&=0,\\
    \dvg \mbu_{(0)} &= 0
  \end{aligned}
\end{equation}
for the velocity $\mbu_{(0)}$. Here, the second order pressure $p_{(2)}$
survives as the incompressible pressure.  

Loosely speaking, the results of
\cite{klainerman-majda,klainermanMajda82} shows that when the initial
data are almost incompressible, the solutions of the compressible
Euler equations \eqref{eq:comp} are good approximations to those of
the incompressible Euler equations \eqref{eq:incomp}. However, in
order to make a more precise statement, and to carry out the analysis
presented later, we recall the following convergence result due to
Schochet \cite{schochet}, and some related results. Note that the
non-dimensionalised, isentropic Euler system \eqref{eq:comp} can be
recast into a non-conservative, evolution form as 
\begin{equation} 
  \label{eq:NondimEuler}
  \Dt U + \mcal{H}(U) + \frac{1}{\veps^2} \mcal{L}(U) = 0,	
\end{equation}
where
\begin{equation}
  \label{eq:Notn1}
  U := \begin{pmatrix}
    \rho\\
    \mbu
  \end{pmatrix}, \
  \mcal{H}(U) := \begin{pmatrix}
    \mbu \cdot \grd \rho \\
    (\mbu \cdot \grd) \mbu 
  \end{pmatrix}, \ 
  \mcal{L}(U) := \begin{pmatrix}
    \veps^2\rho \dvg \mbu \\
    \frac{p^\prime(\rho)}{\rho} \grd \rho
  \end{pmatrix}.	
\end{equation} 
Here, $\mcal{H}$ is the convective operator with a time scale of
order $1$ and $\mcal{L}/\veps^2$ is the acoustic operator with a time
scale of the order of $\veps$. The system \eqref{eq:NondimEuler} is
supplied with periodic boundary conditions
\begin{equation}
  \label{eq:bcNondimEuler}
  U(0,\uu{x})=U_0(\uu{x}), \ \uu{x}\in\mbb{T}^d,
\end{equation}
where $\mbb{T}^d$ denotes the $d$-dimensional torus. 

Following \cite{dellacherie,schochet}, in order to study the solutions
$U=(\rho,\mbu)^T$ of \eqref{eq:comp} in the space $\solspc$,  we
consider the subspaces  
\begin{align}
  {\mcal{E}} &:= \left\{ U \in \solspc \colon \grd \rho = 0, \
               \text{and}  \ \dvg \mbu = 0 \right\}, 	\\  
  \tilde{\mcal{E}} &:= \left\{ U \in \solspc \colon \int_{\mbb{T}^d}
                     \rho d\uu{x} = 0,  \ \text{and}  \ \nabla \times
                     \mbu = 0\right\}. 
\end{align}
Note that $\mcal{E}$ is the subspace of spatially constant densities
and divergence-free velocities. In other words, it is the
`incompressible' subspace and hence, hereafter, $\mcal{E}$ is referred
to as the `well-prepared' space. Having defined $\mcal{E}$ and
$\tilde{\mcal{E}}$, the following Helmholtz-Leray decomposition can be
immediately written:
\begin{equation}
  \label{eq:HelmLer}
  \mcal{E} \oplus \tilde{\mcal{E}} = \solspc,   
  \ \mbox{and } \  \mcal{E} \perp \tilde{\mcal{E}}.
\end{equation}
Thus, for any $U \in \solspc$, there exists a unique $\hat{U} \in
\mcal{E}$ and $\tilde{U}\in \tilde{\mcal{E}}$, such that $U =
\hat{U} + \tilde{U}$. Let us also define the Helmholtz-Leray
projection of any $U \in \solspc$ onto the space $\mcal{E}$ as  
\begin{equation}
  \label{eq:Pdefn}
  \mbb{P}U := \hat{U}.
\end{equation}
In \cite{schochet}, Schochet proved the following crucial estimate
for the solutions of \eqref{eq:comp} in the limit $\veps\to0$, and we
restate this important theorem as done in \cite{dellacherie}. 
\begin{theorem}[\cite{dellacherie,schochet}] 	
  \label{ThmSchochet}
  Let $U$ be a solution of the initial value problem for the system
  \eqref{eq:NondimEuler}, i.e.\ 
  \begin{align}
    \Dt U + \mcal{H}(U) + \frac{1}{\veps^2} \mcal{L}(U) & = 0, \ t>0,
                                                          \ \uu{x}\in\mbb{T}^d,\\ 
    U(0,\uu{x}) &= U_0(\uu{x}), \ \uu{x}\in\mbb{T}^d,
  \end{align}
  and let $\bar{U}$ be a solution of the projected subsystem:  
  \begin{align}
    \Dt \bar{U} + \mbb{P}\mcal{H} (\bar{U}) &= 0, \ t>0, \ 
                                                  \uu{x}\in\mbb{T}^d, \label{eq:Psys}\\ 
    \bar{U}(0,\uu{x}) &= \hat{U}_0(\uu{x}), \ \uu{x} \in
                        \mbb{T}^d,	 \label{eq:PsysIC} 
  \end{align}
  where $\hat{U}_0:=\mbb{P}U_0$. Then, for $\veps \ll 1$, there
  holds the estimate: 
  \begin{equation}
    \norm{\rho_0-\hat{\rho}_0} = \mcal{O}(\veps^2), \ 
    \norm{\mbu_0-\hat{\mbu}_0} = \mcal{O}(\veps) 
    \implies
    \norm{\rho(t)-\bar{\rho}(t)} = \mcal{O}(\veps^2),  \ 
    \norm{\mbu(t)-\bar{\mbu}(t)} = \mcal{O}(\veps) 
    \ \mbox{for all} \ t>0.
    \label{eq:estSchochet}
  \end{equation}    
\end{theorem}
\begin{remark}
  It is easy to see that the projected subsystem
  \eqref{eq:Psys}-\eqref{eq:PsysIC} is equivalent to the incompressible
  Euler system \eqref{eq:incomp}. The essence of the above theorem is
  that when the initial data are well prepared, i.e.\ they are taken
  in $\mcal{E}$, solutions of the compressible Euler system
  \eqref{eq:comp} approximate those of the incompressible Euler
  system \eqref{eq:incomp} when $\veps\to0$. The theorem, clearly, is
  in agreement with the results obtained in \cite{klainerman-majda,
    klainermanMajda82}.    
\end{remark}
\subsection{Zero Mach Number Limit of the Linear Wave Equation System}  
\label{sec:zero-mach-number-1}

In the rest of this section, we consider the linear wave equation
system as a prototype model of the compressible Euler equations
\eqref{eq:NondimEuler}, and present an analogous analysis of the zero
Mach number limit. Linearising the system of equations
\eqref{eq:NondimEuler} about a constant state $(\br,\bu)$, we get the
following wave equation system with advection:   
\begin{equation}
  \label{eq:WaveHL}
  \Dt U + H(U) + \frac{1}{\veps^2} L(U) = 0,	
\end{equation}
where
\begin{align}
  U := \begin{pmatrix}
    \rho\\
    \mbu
  \end{pmatrix} , \ 
  H(U) := \begin{pmatrix}
    \bu \cdot \grd \rho \\
    (\bu \cdot \grd) \mbu 
  \end{pmatrix}, \ 
  L(U) := \begin{pmatrix}
    \veps^2 \br \dvg \mbu \\
    \frac{\ba^2}{\br} \grd \rho
  \end{pmatrix}. \label{eq:Notn2}
\end{align}

Before carrying out an analysis of the solutions of \eqref{eq:WaveHL},
we introduce a few definitions. For any $U_1 = (\rho_1,\mbu_1)$ and
$U_2 = (\rho_2,\mbu_2)$ in $\solspc$, we define the inner-product  
\begin{equation}
  \label{eq:U1U2ip}
  (U_1,U_2):=\frac{\ba^2}{\br \veps^2}\langle\rho_1,\rho_2\rangle+\br
  \langle \mbu_1,\mbu_2\rangle,
\end{equation}
where the inner-products appearing on the right hand side are the usual
$L^2$ inner-products. The norm generated by the above inner-product is
given by  
\begin{equation}
  \label{eq:newNorm}
  \Norm{U} := \left[ \frac{\ba^2}{\br \veps^2} \norm{\rho}^2 + \br
    \norm{\mbu}^2 \right]^\frac{1}{2}.
\end{equation}
Clearly, the norm $\Norm{\cdot}$ is equivalent to the usual norm in
$\solspc$. We also consider the energies
\begin{equation}
  \label{eq:engys_defn}
  E       := \Norm{U}^2,  \ 
  E_{in} := \Norm{\bar{U}}^2,  \ 
  E_{ac} := \Norm{\tilde{U}}^2, 
\end{equation}
which are, respectively, the total energy,  the incompressible energy,
and the acoustic or compressible energy. Since $\mcal{E}
\perp \tilde{\mcal{E}}$, we have $ E = E_{in} + E_{ac}$. In the
following, we state the corresponding linearised version of Schochet's
result, i.e.\ Theorem~\ref{ThmSchochet}. 
\begin{theorem}[\cite{dellacherie}] 
  \label{propLin}
  Let $U$ be a solution of the IVP for the wave equation system
  \eqref{eq:WaveHL}, i.e.\ 
  \begin{align} 
    \Dt U + H(U) + \frac{1}{\veps^2} L(U) & = 0, \ t>0, \
                                            \uu{x}\in\mbb{T}^d, \label{eq:PropEqHL}\\
    U(0,\uu{x}) &= U_0(\uu{x}),	\ \uu{x}\in\mbb{T}^d, \label{eq:PropIcHL}
  \end{align}
  and let $\bar{U}$ be a solution of the IVP: 
  \begin{align}
    \Dt \bar{U} + H(\bar{U}) &= 0, \ t>0, \ \uu{x}\in\mbb{T}^d, \label{eq:PropEqH}\\
    \bar{U}(0,\uu{x}) &= \hat{U}_0(\uu{x}), \ \uu{x}\in\mbb{T}^d, \label{eq:PropIcH} 
  \end{align}
  where $\hat{U}_0:=\mbb{P}U_0$. Let $U=\hat{U}+\tilde{U}$ be the
  Helmholtz-Leray decomposition of $U$. Then the following holds.
  \begin{enumerate}[(i)]
  \item \label{cond1} 
    $\hat{U} = \bar{U}$, 
  \item \label{cond2}
    $\tilde{U}$ is the solution of \eqref{eq:PropEqHL} with
    initial condition $\tilde{U}_0:=(\mbb{I}-\mbb{P})U_0$.
  \end{enumerate}
  Moreover, there holds the energy conservation:
  \begin{equation}
    \label{eq:PropEnCons}
    E_{in}(t) = E_{in}(0)  \ \text{and} \  E_{ac}(t) = E_{ac}(0), \
    \text{for all} \ t>0,
  \end{equation}
  and as a consequence, the following estimate holds:
  \begin{equation}
   \norm{\rho_0-\hat{\rho}_0} = \mcal{O}(\veps^2), \ 
   \norm{\mbu_0-\hat{\mbu}_0} = \mcal{O}(\veps) 
   \implies
   \norm{\rho(t)-\bar{\rho}(t)} = \mcal{O}(\veps^2),  \ 
   \norm{\mbu(t)-\bar{\mbu}(t)} = \mcal{O}(\veps) 
   \ \mbox{for all} \ t>0.
    \label{eq:estSchochetLin}
  \end{equation}
\end{theorem}
\begin{remark}
  Carrying out an analogous asymptotic analysis using the ansatz
  \eqref{eq:f_ansatz} for the wave equation yields the system:
  \begin{align}
    \Dt \mbu_{(0)}+(\bu\cdot\nabla)
    \mbu_{(0)}+\frac{\ba^2}{\br}\nabla\rho_{(2)}&=0, \label{eq:waveU0} \\
    \nabla\cdot \mbu_{(0)}&=0 \label{eq:waveDivU0}
  \end{align}
  for the unknowns $(\rho_{(2)},\mbu_{(0)})$, analogous to
  \eqref{eq:incomp}. Since the boundary conditions are periodic, and
  \eqref{eq:waveU0} is linear, the divergence-free condition
  \eqref{eq:waveDivU0} on $\mbu_{(0)}$ now forces $\rho_{(2)}$ to be a
  constant. Hence, $\rho\equiv\const$, and \eqref{eq:waveU0} then
  reduces to a linear transport equation for $\mbu_{(0)}$, and the
  divergence-free condition \eqref{eq:waveDivU0} is now required only
  at time $t=0$. Thus, the projected subsystem
  \eqref{eq:PropEqH}-\eqref{eq:PropIcH} and the system 
  \eqref{eq:waveU0}-\eqref{eq:waveDivU0} are equivalent, and are
  incompressible.  
\end{remark}

Theorem~\ref{propLin} lies at the centre of the analysis of numerical
schemes presented in \cite{dellacherie}. However, numerical
discretisation procedures introduce numerical diffusion, dispersion or
higher order derivative terms, depending on the order of the scheme
under consideration. Hence, the theorem has to be generalised to
accomodate more general spatial differential operators arising from
the discretisation. Such a generalisation is presented in
\cite{dellacherie} which also gives a sufficient condition to ensure
the estimate \eqref{eq:estSchochetLin}.     
   
\begin{theorem}[\cite{dellacherie}]
  \label{ThmDel}
  Let $U$ be a solution of the IVP:
  \begin{align} 
    \Dt U + \mcal{F}_{\uu{x}} U &= 0, \ t>0, \
    \uu{x}\in\mbb{T}^d, \label{eq:mThmEq}  \\
    U(0,\uu{x}) &= U_0(\uu{x}), \ \uu{x}\in\mbb{T}^d, \label{eq:mThmIc} 
  \end{align}
  which is assumed to be well-posed in
  $L^\infty\left([0,\infty);\solspc\right)$, with $\mcal{F}_{\uu{x}}$ a
  linear spatial differential operator. Then the following conclusions
  hold.
  \begin{enumerate}[(i)]
  \item \label{ThmDelCond1}
    The solution $U$ satisfies the estimate 
    \begin{equation} \label{eq:ThmDelEst1}
      \norm{U_0-\mbb{P}U_0} = \mcal O(\veps) \implies
      \norm{U(t)-\bar{U}(t)} = \mcal O(\veps), \ \text{for all} \ t>0,  
    \end{equation}
    where $\bar{U}$ is a solution of \eqref{eq:mThmEq} with the initial
    condition $\bar{U}(0) = \mbb{P}U_0$. However, we don't have the 
    apriori estimate $\norm{U(t) - \mbb{P}U(t)} = \mcal O(\veps)$ for
    all $t>0$.
  \item \label{ThmDelCond2}
    When the operator $\mcal{F}_{\uu{x}}$ leaves $\mcal{E}$ invariant, i.e.\
    whenever $U_0\in\mcal{E}$ implies $U(t)\in\mcal{E}$ for all $t>0$,
    then $U$ satisfies the estimate \eqref{eq:ThmDelEst1}, and in
    addition we have 
    \begin{equation} 
      \label{eq:ThmDelEst2}
      \norm{U_0-\mbb{P}U_0} = \mcal O(\veps) \implies \norm{U(t) -
        \mbb{P}U(t)} = \mcal O(\veps)  \ \text{for all} \ t>0.
    \end{equation}
  \end{enumerate}
\end{theorem}
\begin{remark}
  Note that in Theorem~\ref{ThmSchochet} and \ref{propLin} the
  estimate on the difference $\norm{\rho-\bar{\rho}}$ is
  $\mcal{O}(\veps^2)$, whereas in \cite{dellacherie} a scaled
  variable is used instead of $\rho$, and hence the corresponding
  estimate is $\mcal{O}(\veps)$.  
\end{remark}
\section{Asymptotic Preserving and Asymptotically Accurate Schemes} 
\label{sec:sec3}

As mentioned in the introduction, the AP property must be necessary
for any numerical scheme to survive in the passage of the limit
$\veps\to0$. However, it is wellknown from the literature, e.g.\
\cite{dellacherie, dellacherie-omnes-rieper, guillard}, that the
accuracy of a scheme can also deteriorate in the low Mach number
limit. Hence, it is essential for an AP scheme to be not only
consistent with the stiff limit but also to maintain its accuracy. In
addition, an AP scheme is expected to be accurate in the transient
regimes as well. The two primary aims of the present work are to
design and analyse, and test an IMEX-RK finite volume scheme which is
AP while maintaining the accuracy uniformly in $\veps$. In what
follows, we formally state the definitions of the AP property, and the
asymptotic accuracy (AA). The key principle behind the AA property is
an estimate which states that it is essential for a numerical solution
to be $\veps$-close to the incompressible solution as in the continuous
asymptotics. Following Jin \cite{Jin12}, we show how the AA property of
a scheme suffices to obtain the error estimates expected from an AP
scheme.              

\subsection{AP Property}
Numerical solution of singular perturbation problems is, in general, a
difficult task mainly because of the existence of multiple scales,
the dependence of the stability characteristics on the perturbation
parameters, and the transitional nature of the governing equations. A
classical numerical approximation scheme may not resolve the scales,
its stability requirements might deteriorate, and in the singular
limit the scheme might approximate a completely different set of
equations than the actual limiting system. AP schemes provide a robust
approach as a remedy against these problems. AP schemes were
originally developed in the context of kinetic transport equations;
see \cite{Degond13, jin-ap,Jin12}, and the references therein for more
details.  
\begin{definition}
\label{defn:AP}
  Let $\mcal{P}^{\veps}$ denote a singularly perturbed problem with the
  perturbation parameter $\veps$. Let $\mcal{P}^0$ denote the limiting
  system of $\mcal{P}^\veps$ when $\veps \to 0$. A discretisation
  $\mcal{P}_{h}^{\veps}$ of  $\mcal{P^{\veps}}$, with $h$ being the
  discretisation parameter, is called AP if 
  \begin{enumerate}[(i)]
    \item $\mcal{P}^{0}_h$ is a consistent discretisation of the
      problem $\mcal{P}^0$, called the asymptotic consistency, and
    \item the stability constraints on $h$ are independent of $\veps$,
      called the asymptotic stability.
  \end{enumerate} 
In other words, the following diagram commutes:
\begin{equation*}
    \begin{CD}
      \mcal{P}^{\veps}_{h} @>{h\to 0}>> \mcal{P}^{\veps}\\
      @VV{\veps\to 0}V  @VV{\veps\to 0}V \\
      \mcal{P}^{0}_{h} @>{h\to 0}>> \mcal{P}^0
    \end{CD} 
  \end{equation*}
\end{definition}
Jin in \cite{Jin12} has presented a systematic discussion regarding the
error estimates expected from an AP scheme. The solutions of the
continuous systems $\mcal{P}^{\veps}$ and  $\mcal{P}^{0}$ satisfy the
following error estimate:
\begin{equation} \label{eq:error_con_estimate}
\norm{\mcal{P}^{\veps} - \mcal{P}^{0}} = \mcal{O} (\veps).
\end{equation}
Suppose $\mcal{P}_{h}^{\veps}$ is a $p^{th}$-order accurate
approximation to $\mcal{P}^{\veps}$ for a fixed $\veps$ with
discretisation parameter $h$. Caused by the presence of the
singular perturbation parameter $\veps$, a classical numerical
approximation is typically expected to give the following error
estimate: 
\begin{equation}\label{eq:error_class_estimate}
e_{class} := \norm{\mcal{P}_{h}^{\veps} - \mcal{P}^{\veps} } = \mcal{O}
(h^p / \veps^s) , \ \mbox{for} \ 1 \leq s \leq p.
\end{equation}
If the scheme $\mcal{P}_{h}^{\veps}$ is  assumed to be AP, it may
give the following estimates 
\begin{align}
\norm{\mcal{P}_{h}^{\veps} - \mcal{P}_{h}^{0} }  &= \mcal{O} (\veps), \
\mbox{uniformly in} \ h \ \mbox{and} , \label{eq:error_ap_estimate_eps}\\
\norm{\mcal{P}_{h}^{0} - \mcal{P}^{0} }  &= \mcal{O}
                                           (h^p) \label{eq:error_ap_estimate_h}.
\end{align} 
Combining \eqref{eq:error_con_estimate},
\eqref{eq:error_ap_estimate_eps} and \eqref{eq:error_ap_estimate_h},
by using the triangle inequality, for an AP scheme we have the error
estimate:   
\begin{equation}\label{eq:error_ap_estimate}
e_{ap} := \norm{\mcal{P}_{h}^{\veps} - \mcal{P}^{\veps} } = \mcal{O}
(\veps + h^p) .
\end{equation}
Comparing the estimates \eqref{eq:error_ap_estimate} and
\eqref{eq:error_class_estimate}, it can be concluded that unlike a
classical scheme, the error $e_{ap}$ of an AP scheme remains 
bounded as $\veps \to 0$. The estimate
\eqref{eq:error_ap_estimate_eps}, which plays a crucial role to obtain
the error bound \eqref{eq:error_ap_estimate}, makes sure that the
numerical solution of the singular perturbation problem
$\mcal{P}^{\veps}$ remains $\veps$-close to the numerical solution of
the limit problem $\mcal{P}^{0}$, which is a property exhibited by the
exact solutions of $\mcal{P}^{\veps}$ and $\mcal{P}^{0}$. 

In a different context, Dellacherie in \cite{dellacherie} has
investigated the origin of inaccuracies exhibited by explicit
Godunov-type schemes, arising from the creation of spurious acoustic 
waves. The results of the above study reveal that the inaccuracies 
originate due to the inability of the schemes to maintain the estimate
\eqref{eq:error_ap_estimate_eps} even if the initial datum is
well-prepared. It has been shown in \cite{dellacherie}, cf.\ also 
Theorem~\ref{ThmDel}, that the $\mcal{E}$-invariance is a sufficient
condition for linear schemes to satisfy the estimate
\eqref{eq:error_ap_estimate_eps}, provided the initial datum is
well-prepared.  
\subsection{Asymptotic Accuracy}
\label{sec:aa-schemes}
A numerical scheme for the compressible Euler system should possess
the ability to maintain the solutions close to the incompressible
solutions in $\mcal{E}$ for all times, whenever the initial data are
well-prepared. It has to be noted that Theorem~\ref{ThmDel}, and its
conclusions are valid only for a linear scheme applied to the wave
equation system. A scheme maintaining the estimate
\eqref{eq:ThmDelEst2} from Theorem~\ref{ThmDel} was designated 
to be a low Mach accurate scheme in \cite{dellacherie}, and the
$\mcal{E}$-invariance is a sufficient condition at least in
the linear case. It should be taken into account that the estimates
\eqref{eq:ThmDelEst2} and \eqref{eq:error_ap_estimate_eps} are the
same in the case of the wave equation. Therefore, motivated by this
result, the discussion on AP schemes, and analogous considerations
from \cite{dellacherie}, we propose the asymptotic accuracy of a
scheme to be its ability to preserve the well-prepared space
$\mcal{E}$. 
\begin{definition}
  \label{defn:AA}
  A numerical approximation for the compressible Euler system
  \eqref{eq:comp} is said to be asymptotically accurate (AA), if it
  leaves the incompressible subspace $\mcal{E}$ invariant.  
\end{definition}
\begin{remark}
  Even though the AA  property is a sufficient condition to obtain the
  estimate \eqref{eq:error_ap_estimate_eps}, it is very useful as it
  gives an easy and verifiable criterion for the wave equation system,
  leading towards the AP property.  
\end{remark}
The estimate \eqref{eq:error_ap_estimate_h} for an AP scheme gives
$p^{th}$-order accuracy in the stiff limit $\veps\to0$. Thus, if the
solutions of the modified partial differential equation (MPDE)
resulting from a $p^{th}$-order linear scheme applied to the wave
equation leaves $\mcal{E}$ invariant, then such a scheme will avoid
inaccuracies. In addition, the scheme will reduce to a $p^{th}$-order
accurate discretisation of the limit system in the stiff limit.
\begin{remark}
  It has to be noted both the AP and AA properties are to be satisfied
  by the time semi-discrete as well as space-time fully-discrete
  schemes. In the following sections we establish that the IMEX-RK
  schemes considered in this paper possess both the AP and AA
  properties.    
\end{remark}
    
\section{Time Semi-discrete Scheme}  
\label{sec:sec4}
\subsection{Implicit-Explicit (IMEX) Runge-Kutta (RK) Time
  Discretisation} 
The IMEX-RK schemes are designed for the numerical integration of
stiff systems of ordinary differential equations (ODEs) of the form
\begin{equation}
  \label{eq:stiff_ODE}
  y^\prime = f(t,y) + \frac{1}{\veps} g(t,y),
\end{equation}
where $y\colon\mbb{R}\to\mbb{R}^n, \
f,g\colon\mbb{R}\times\mbb{R}^n\to\mbb{R}^n$, and $0<\veps\ll 1$ is
usually known as the stiffness parameter. The functions $f$ and $g$
are called, respectively, the non-stiff part and the stiff part of the
system \eqref{eq:stiff_ODE}. Stiff systems of ODEs of the type
\eqref{eq:stiff_ODE} typically arise in the modelling of semiconductor
devices, study of kinetic equations, theory of hyperbolic systems with
relaxation etc.; see, e.g.\ \cite{hairer-wanner-2}, for a comprehensive
treatment of such systems.

Let $y^n$ be a numerical solution of \eqref{eq:stiff_ODE} at
time $t^n$, and let $\Dlt$ denote a fixed timestep. An $s$-stage
IMEX-RK scheme, cf.,\ e.g.\ \cite{AscherRuuthSpiteri,
  pareschi-russo-conf}, updates $y^n$ to $y^{n+1}$ through $s$
intermediate stages:  
\begin{align}
  Y_i &= y^n + \Dlt \sum\limits_{j=1}^{i-1}\tilde{a}_{i,j} f(t^n + \tilde{c}_j\Dlt, Y_j) + 
        \Dlt\sum \limits_{j=1}^s a_{i,j} \frac{1}{\veps}g(t^n + c_j
        \Dlt,Y_j), \ 1 \leq i \leq s, \label{eq:imex_Yi} \\
  y^{n+1} &= y^n  + \Dlt \sum\limits_{i=1}^{s}\tilde{\omega}_{i} f(t^n
            + \tilde{c}_i\Dlt, Y_i) + \Dlt\sum \limits_{i=1}^s
            \omega_{i}\frac{1}{\veps} g(t^n +
            c_i\Dlt,Y_i). \label{eq:imex_yn+1} 
\end{align}
Let us denote $\tilde{A} = (\tilde{a}_{i,j})$, $ A= (a_{i,j}) $,  $\tilde{c} =
(\tilde{c}_1, \tilde{c}_2,\ldots,\tilde{c}_s), \ c = (c_1, c_2, \ldots,
c_s), \ \tilde{\omega}= (\tilde{\omega}_1, \tilde{\omega}_2,
\ldots,\tilde{\omega}_s)$ and $\omega = (\omega_1, \omega_2, \ldots 
,\omega_s)$. The above IMEX-RK scheme
\eqref{eq:imex_Yi}-\eqref{eq:imex_yn+1} can be symbolically represented by
the double Butcher tableau: 
\begin{figure}[htbp]
  \centering
  \begin{tabular}{c|c}
    $\tilde{c}^T$	&$\tilde{A}$\\
    \hline 
			&$\tilde{\omega}^T$
  \end{tabular}
  \quad
  \begin{tabular}{c|c}
    $c^T$	&$A$\\
    \hline 
		&$\omega^T$
  \end{tabular}
  \caption{Double Butcher tableau of an IMEX-RK scheme.}
  \label{fig:butcher_tableau}
\end{figure}

The matrices $A$ and $\tilde{A}$ are $ s\times s$ matrices such that
resulting scheme is explicit in $f$  and implicit in $g$. However, in
order to reduce the number of implicit evaluations in the
intermediate stages \eqref{eq:imex_Yi}, we consider only the
so-called diagonally implicit Runge-Kutta (DIRK) schemes in which
$\tilde{a}_{i,j}=0$ for $ j \geq i$, and $a_{i,j}=0$ for $j>i$.  The
coefficients $\tilde{c}_i$ and $c_i$, and the weights
$\tilde{\omega}_i$ and $\omega_i$ are fixed by the order conditions;
see, e.g.\ \cite{hairer-wanner-2,pareschi-russo-conf, pareschi-russo},
for details. For the sake of completion, and to draw reference for the
analysis carried out later, we write down the order conditions for a
first order and second order IMEX-RK scheme as follows: 
\begin{gather}
  \tilde{c}_i = \sum_{j=1}^{i-1}\tilde{a}_{i,j}, \quad c_i = \sum_{j=1}^i
  a_{i,j}, \label{eq:IMEX-RK-consistency}\\ 
\sum_{i=1}^s \tilde{\omega}_i = 1, \quad \sum_{i=1}^s \omega_i =
1, \label{eq:IMEX-RK-1storder} \\ 
\sum_{i=1}^s \tilde{\omega}_i \tilde{c}_i = \frac{1}{2},  \quad
\sum_{i=1}^s \omega_i c_i = \frac{1}{2}, \quad 
\sum_{i=1}^s \tilde{\omega}_i c_i = \frac{1}{2}, \quad
\sum_{i=1}^s \omega_i\tilde{c}_i =
\frac{1}{2}. \label{eq:IMEX-RK-2ndorder}  
\end{gather}
The conditions \eqref{eq:IMEX-RK-consistency},
\eqref{eq:IMEX-RK-1storder} and \eqref{eq:IMEX-RK-2ndorder} are,
respectively, the consistency, the first-order, and the second-order
order conditions.
\begin{definition}
\label{eq:GSA}
An IMEX-RK scheme with the Butcher tableau given in
Figure~\ref{fig:butcher_tableau} is said to be globally stiffly
accurate (GSA), if
\begin{equation} 
  \tilde{a}_{s,j} = \tilde{\omega}_{j},  \quad a_{s,j} = \omega_{j}  \ \mbox{for all} \  j = 1,
\ldots, s.  
\end{equation}
\end{definition}
 To further simplify the analysis of the schemes
presented in this paper, we restrict ourselves only to two types of
DIRK schemes, namely the type-A and type-CK schemes which are defined
below; see \cite{boscarino-sinum2007,KENNEDY2003139} for more details.
\begin{definition}
\label{eq:typeA_CK}
An IMEX-RK scheme with the Butcher tableau given in
Figure~\ref{fig:butcher_tableau} is said to be of 
\begin{itemize}
\item type-A, if the matrix $A$ is invertible; 
\item type-CK, if the matrix $A \in \mbb{R}^{s \times s}, \ s \geq 2$,
  can be written as  
\begin{equation*}  
  A = 
  \begin{pmatrix}
    0 & 0 \\
    \alpha & A_{s-1 }
  \end{pmatrix},
\end{equation*}
where $\alpha \in \mbb{R}^{s-1} $ and $A_{s-1} \in \mbb{R}^{s-1 \times
s-1}$ is invertible.
\end{itemize} 
\end{definition}
In the results presented in the later sections, we have used
predominantly the first order Euler(1,1,1), and second order
ARS(2,2,2) schemes for time discretisations; see the Appendix for the
Butcher tableaux of the several variants of the IMEX-RK schemes we
have considered in our numerical case studies. Here, in the triplet
$(s,\sigma,p)$, $s$ is the number of stages of the implicit part, the
number $\sigma$ gives the number of stages for the explicit part and
$p$ gives the overall order of the scheme. For a detailed account of
IMEX-RK schemes we refer the reader to
\cite{AscherRuuthSpiteri,boscarino-sinum2007,KENNEDY2003139,
  pareschi-russo-conf, pareschi-russo} and the references therein.     

\subsection{IMEX-RK Time Discretisation of the Euler System} 
Based on the asymptotic analysis and the convergence results presented
in Section~\ref{sec:sec2}, we can split the flux functions in the Euler
system \eqref{eq:comp} into a stiff and a non-stiff part. Denoting by
$W=(\rho,\mbq)^T$, the vector of conservative variables, the fluxes
are split via 
\begin{equation}
  \label{eq:flux_split}
  G(W) := \begin{pmatrix}
    \mbq \\
    \frac{p(\rho)}{\veps^2}\mbb{I}_2 
  \end{pmatrix}, \quad 
  F(W) := \begin{pmatrix}
    \uu{0} \\
    \frac{\mbq \otimes \mbq}{\rho}
  \end{pmatrix},
\end{equation}
where $\mbq = \rho\mbu$ is the momentum. Applying the IMEX-RK time 
discretisation, i.e.\ treating $F$ explicitly and $G$ implicitly,
results in the following semi-discrete scheme.  
\begin{definition}
The $k^{th}$ stage of an $s$-stage IMEX-RK scheme for the Euler system 
\eqref{eq:comp} is defined as
\begin{align}
  \rho^{k} &= \rho^{n} - \Dlt a_{k,l} \nabla\cdot \mbq^{l}, \label{eq:rho_ieu_sd_k} \\  
  \mbq^{k} &=  \mbq^{n} - \Dlt \tilde{a}_{k,\ell} \nabla\cdot\left(\frac{\mbq^{\ell} \otimes
            \mbq^{\ell}}{\rho^{\ell}}\right) - \Dlt  a_{k,l}
            \frac{\nabla p\left(\rho^{l}\right)}{\veps^2},  \  k = 1,\ldots, s.
                         \label{eq:rhou_ieu_sd_k}
\end{align}  
The numerical solution at time $t^{n+1}$ is given by 
\begin{align}
  \rho^{n+1} &= \rho^{n} - \Dlt \omega_k \nabla\cdot \mbq^{k},
               \label{eq:rho_ieu_sd_n+1} \\ 
  \mbq^{n+1} &=  \mbq^{n} - \Dlt
            \tilde{\omega}_{k} \nabla\cdot\left(\frac{\mbq^{k} \otimes \mbq^{k}}{\rho^{k}}
            \right) - \Dlt \omega_{k} \frac{\nabla
            p\left(\rho^{k}\right)}{\veps^2}. \label{eq:rhou_ieu_sd_n+1}  
\end{align}
\end{definition}
In the above, and throughout the rest of this paper, we follow the
convention that a repeated index always denotes the summation with
respect to that index. The indices $\ell$ and $l$ are used to denote,
respectively, the summation in the explicit and implicit terms, i.e.\ they
assume values in the sets $\{1,2,\ldots,k-1\}$ and $\{1,2,\ldots,k\}$.  

Though the intermediate stages
\eqref{eq:rho_ieu_sd_k}-\eqref{eq:rhou_ieu_sd_k}  consist of two
implicit steps, the resolution of the scheme
\eqref{eq:rho_ieu_sd_k}-\eqref{eq:rhou_ieu_sd_k} is quite simple. As in
\cite{degond_tang}, we eliminate $q^{k}$ between
\eqref{eq:rho_ieu_sd_k} and \eqref{eq:rhou_ieu_sd_k}, and obtain the
nonlinear elliptic equation:
\begin{equation}
  -\frac{\Dlt^2 a_{k,k}^2}{\veps^2} \Delta p\left(\rho^{k}\right)+
  \rho^{k} = \hat{\rho}^{k} - \Dlt a_{k,k} \nabla \cdot
  \hat{\mbq}^{k} \label{eq:rho_ieu_k_elliptic} 
\end{equation}
for $\rho^{k}$. Here, we have denoted by
\begin{align}
\hat{\rho}^{k} &:= \rho^n - \Dlt a_{k,\ell} \nabla \cdot
                   \mbq^{\ell},  \label{eq:rhohat_ieu_sd_k}\\  
  \hat{\mbq}^{k} &:= \mbq^n - \Dlt \tilde{a}_{k,\ell} \nabla \cdot
                  \left(\frac{\mbq^{\ell} \otimes \mbq^{\ell}}{\rho^{\ell}}\right) - \Dlt
                  \frac{a_{k,\ell}}{\veps^2} \nabla
                  p\left(\rho^{\ell}\right)\label{eq:rhouhat_ieu_sd_k}   
\end{align}
those terms that can be explicitly evaluated. Once $\rho^{k}$ is
known after solving the elliptic equation
\eqref{eq:rho_ieu_k_elliptic}, $\mbq^{k}$ can be explicitly evaluated 
from \eqref{eq:rhou_ieu_sd_k}. Finally, the updates $\rho^{n+1}$ and
$\mbq^{n+1}$ are calculated explicitly using
\eqref{eq:rho_ieu_sd_n+1}-\eqref{eq:rhou_ieu_sd_n+1} with the values
obtained from the intermediate stages; see also
\cite{boscarino_weno_jcp_19, degond_tang,dimarco_loubere_jcp_2018,
  dimarco_loubere_siamjsc_2017,tang_kinrelmod_2012} for related
approaches.  
\begin{remark} 
  In order to further reduce the nonlinear nature of the elliptic
  equation \eqref{eq:rho_ieu_k_elliptic}, and to make the numerical
  implementation simpler, in our computations we transform the
  nonlinear elliptic equation for $\rho^{k}$ into a semi-linear
  elliptic equation for the pressure $p^{k}:=p(\rho^{k})$.  
\end{remark}

\section{Analysis of the Time Semi-discrete Scheme} 
\label{sec:sec5} 
The goal of this section is to establish the AP property and
asymptotic accuracy of the time semi-discrete IMEX-RK scheme
\eqref{eq:rho_ieu_sd_k}-\eqref{eq:rhou_ieu_sd_n+1} in sense of
Definitions~\ref{defn:AP} and \ref{defn:AA}.

\subsection{AP Property}
Establishing the AP property consists of showing the consistency of
the scheme with that of the incompressible system as $\veps\to0$, with
its stability requirements independent of $\veps$. In order to show
the former, we perform an asymptotic analysis of the time
semi-discrete scheme
\eqref{eq:rho_ieu_sd_k}-\eqref{eq:rhou_ieu_sd_n+1}, and to show the
latter, we use a linear stability analysis. 

\subsubsection{Asymptotic Consistency}
\label{sec:asympt-cons}
\begin{theorem}
  \label{Asym_cons_SD}
  Suppose that the data at time $t^{n}$ are well-prepared, i.e.\
  $\rho^{n}$ and $\mbu^{n}$ admit the decomposition: 
  \begin{align}
    \rho^n &= \rho_{(0)}^n + \veps^2 \rho_{(2)}^n , \label{eq:rho_wp_semi-disc} \\ 
    \mbu^n &= \mbu_{(0)}^n + \veps \mbu_{(1)}^n, \label{eq:u_wp_semi-disc} 
  \end{align}
  where $\nabla \rho^{n}_{(0)} = 0$ and $\nabla \cdot \mbu^{n}_{(0)} =
  0$. Then for a GSA scheme, if the intermediate solutions $\rho^{k}$
  and $\mbu^{k}$ defined by \eqref{eq:rho_ieu_sd_k} and
  \eqref{eq:rhou_ieu_sd_k} admit the decomposition
  \eqref{eq:f_ansatz}, then they must satisfy $\nabla \rho^{k}_{(0)} =
  0$ and $\nabla \cdot \mbu^{k}_{(0)}=0$, i.e.\ $\rho^k$ and $\mbu^k$
  are well-prepared as well. As a consequence, if the numerical
  solutions $\rho^{n+1}$ and $\mbu^{n+1}$ admit the decomposition
  \eqref{eq:f_ansatz}, then they are also well-prepared. In other
  words, the semi-discrete scheme
  \eqref{eq:rho_ieu_sd_k}-\eqref{eq:rhou_ieu_sd_n+1} is asymptotically
  consistent with the incompressible limit system in the sense of
  Definition~\ref{defn:AP}.  
\end{theorem} 
 \begin{proof}
   We use induction on the number of stages to prove the theorem. To
   begin with, we prove the consistency for the first stage, i.e.\ $k=1$,
   which corresponds to a fully implicit step. We plugin the ansatz
   \eqref{eq:f_ansatz} for each of the dependent variables in
   \eqref{eq:rho_ieu_sd_k}-\eqref{eq:rhou_ieu_sd_n+1}. Equating to zero
   the $\mcal{O}{(\veps^{-2})}$ terms in system
   \eqref{eq:rho_ieu_sd_k}-\eqref{eq:rhou_ieu_sd_k} yields
   \begin{equation}
     \nabla p \left( \rho^{1}_{(0)} \right) = 0 \label{eq:ieu_grad_rho_1_0} .
   \end{equation}
   Hence, the zeroth order density $\rho^{1}_{(0)}$ is spatially
   constant. In an analogous way, we can show that the first order density
   $\rho^{1}_{(1)}$ is also spatially constant. 

   Next, we consider  the $\mcal{O}(1)$ terms in the mass update
   \eqref{eq:rho_ieu_sd_k} to obtain  
   \begin{equation}
     \label{eq:ieu_mass_0_0}
     \frac{\rho^{1}_{(0)} - \rho^{n}_{(0)}}{\Dlt} = - a_{11} \rho^{1}_{(0)} \nabla \cdot
     \mbu^{1}_{(0)}.
   \end{equation}
   We Integrate the above equation \eqref{eq:ieu_mass_0_0} over a
   spatial domain $\Omega$ and use the Gauss' divergence theorem to
   get 
   \begin{equation}
     \label{eq:ieu_mass_0_0_inte}
     \abs{\Omega} \frac{\rho^{1}_{(0)} - \rho^{n}_{(0)}}{\Dlt} = - a_{11}
     \rho^{1}_{(0)} \int_{\Omega}\nabla \cdot \mbu^{1}_{(0)}d\uu{x} = - a_{11}
     \rho^{1}_{(0)} \int_{\D \Omega} \uu{\nu} \cdot \mbu^{1}_{(0)}d\sigma.  
   \end{equation}
   If the boundary conditions of the problem are periodic or wall,
   then the right most integral in the above equation vanishes,
   yielding $\rho^{1}_{(0)}=\rho^{n}_{(0)}$. As the zeroth order
   density $\rho^{n}_{(0)}$ is a constant, $\rho^{1}_{(0)}$ as well as
   $\rho^{1}_{(1)}$ are constants. Using this in equation
   \eqref{eq:ieu_mass_0_0}, we obtain the divergence condition:
   \begin{equation}
     \nabla \cdot \mbu^{1}_{(0)} = 0
   \end{equation} 
   for the leading order velocity $\mbu^{1}_{(0)}$. This completes the proof
   for $k = 1$. 

   To prove the result for $k = 2$ onwards, we rewrite the intermediate
   stages \eqref{eq:rho_ieu_sd_k}-\eqref{eq:rhou_ieu_sd_k} in the form
   \begin{align}
     \rho^{k} &= \hat{\rho}^k -\Dlt a_{k,k} \dvg
                \mbq^{k}, \label{eq:ieu_rho_k^_semi_disc} \\
     \mbq^{k}&=\hat{\mbq}^k - \Dlt \frac{a_{k,k}}{\veps^2}\nabla
            p(\rho^{k}). \label{eq:ieu_rhou_k^_semi_disc} 
   \end{align}
   Proceeding as in the case of $k = 1$, we can obtain $\rho^{k}_{(0)} =
   \rho^{n}_{(0)} $. Hence, for the $k^{th}$ stage, the zeroth order density
   $\rho^{k}_{(0)}$ is same as that of $\rho^n_{(0)}$, and in turn we also obtain
   the divergence condition: 
   \begin{equation}
     \nabla \cdot \mbu^{k}_{(0)} = 0
   \end{equation} 
   for $\mbu^{k}_{(0)}$.

   Since the scheme under consideration is GSA, the numerical
   solution at time $t^{n+1}$ is same as the solution at the $s^{th}$ stage. 
      Summarising the above steps, the asymptotic limit of scheme
   \eqref{eq:rho_ieu_sd_k}-\eqref{eq:rhou_ieu_sd_n+1} is given by
   \begin{equation}
     \label{eq:ieu_imex_SD_lim}
     \begin{aligned}
       \rho^{n+1}_{(0)} &= \const,\\ 
       \mbu^{n+1}_{(0)} &= \mbu^{n}_{(0)} - \Dlt \tilde{\omega}_{k} \nabla
       \cdot \left(\mbu^k_{(0)} \otimes \mbu^k_{(0)}\right) - \Dlt
       \omega_{k} \nabla p^{k}_{(2)},\\ 
       \nabla \cdot \mbu^{n+1}_{(0)} &= 0.  
     \end{aligned}
   \end{equation}
   Clearly, \eqref{eq:ieu_imex_SD_lim} is a consistent discretisation
   of the incompressible limit system \eqref{eq:incomp}. Hence, the
   semi-discrete scheme 
   \eqref{eq:rho_ieu_sd_k}-\eqref{eq:rhou_ieu_sd_n+1} is
   asymptotically consistent.    
\end{proof}
\begin{remark}
  Note that the above proof is valid for any $s$-stage GSA IMEX-RK
  scheme. However, it uses the fact that $a_{k,k} \neq 0 $ for
  $k = 1, 2,\ldots, s$, which is the property of an RK scheme of
  type-A. The result also holds for a type-CK GSA scheme in which the
  first step is trivial, and then on the proof follows similar lines
  as that of a type-A scheme.
\end{remark}

\subsubsection{Linearised $L^2$-stability Analysis}
\label{sec:line-l2-stab}

In this subsection, we analyse the correction terms arising from the
time discretisation, and their effect on the asymptotic stability of the
resulting scheme. In order to analyse these correction terms, we
follow the standard MPDE approach; see also \cite{guillard} for a
related analysis on stability.    

In order to establish the asymptotic stability of the IMEX-RK time
discretisation for the isentropic Euler system \eqref{eq:comp}, we
carry out a thorough linear $L^2$-stability analysis for the wave
equation system \eqref{eq:WaveHL} as a linearised model. As a result,
we obtain a sufficient condition for $L^2$-stability; see
\cite{ArunDasGuptaSamHyp2016} for the analysis of a first order IMEX
scheme. We believe that the results of linear $L^2$-stability analysis
holds good also for the nonlinear Euler system
\eqref{eq:comp}. Analogous to Section~\ref{sec:sec4}, an IMEX-RK time
discrete scheme for the linear wave equation system \eqref{eq:WaveHL}
is defined as follows.     
\begin{definition}
The $k^{th}$ stage of an $s$-stage IMEX-RK scheme for the wave
equation system \eqref{eq:WaveHL} is given by
\begin{align}
  \rho^{k} &= \rho^n - \Dlt \tilde{a}_{k,\ell}(\bu\cdot\nabla)\rho^{\ell}
               -\Dlt a_{k,l}\br\dvg \mbu^{l}, \label{eq:weq_rho_k_semi_disc} \\
  \mbu^{k}&=\mbu^n-\Dlt \tilde{a}_{k,\ell} (\bu\cdot\nabla)\mbu^{\ell} 
           - \Dlt a_{k,l}\frac{\ba^2}{\br\veps^2}\nabla\rho^{l}.
           \label{eq:weq_u_k_semi_disc}
\end{align}
The numerical solutions $\rho^{n+1}$ and $\mbu^{n+1}$ at time $t^{n+1}$
are defined as 
\begin{align}
  \rho^{n+1} &= \rho^n - \Dlt \tilde{\omega}_k (\bu\cdot\nabla)\rho^k
               -\Dlt \omega_k \br\dvg
               \mbu^{k}, \label{eq:weq_rho_semi_disc_finupd} \\ 
  \mbu^{n+1}&=\mbu^n-\Dlt \tilde{\omega}_k  (\bu\cdot\nabla)\mbu^k 
           - \Dlt \omega_k \frac{\ba^2}{\br\veps^2}\nabla\rho^{k}.
           \label{eq:weq_u_semi_disc_finupd} 
\end{align}
\end{definition}
In the following, we derive the MPDE corresponding to a general second
order accurate time discrete scheme of the form
\eqref{eq:weq_rho_k_semi_disc}-\eqref{eq:weq_u_semi_disc_finupd}. We 
show that the solution of the MPDE is energy-dissipative under a
sufficient condition involving only the RK coefficients, and hence,
the timesteps are independent of $\veps$; see also
\cite{ArunDasGuptaSamHyp2016}. Thus, the time-discrete scheme
\eqref{eq:rho_ieu_sd_k}-\eqref{eq:rhou_ieu_sd_n+1} for the Euler
system is linearly asymptotically $L^2$-stable.
\begin{theorem}
  \label{ieu_Lin_stab_an}
  Consider a second-order time-discrete IMEX-RK scheme of the form
  \eqref{eq:rho_ieu_sd_k}-\eqref{eq:rhou_ieu_sd_n+1}, and let the
  coefficients $b^{(4)}_1,b^{(4)}_2,b^{(4)}_3$ and $b^{(4)}_4$ be
  defined by \eqref{eq:def_b_4} in the Appendix. Then, the scheme is
  linearly $L^2$-stable as long as the timesteps are bounded, and the
  coefficients $b^{(4)}_1,b^{(4)}_2,b^{(4)}_3$ and $b^{(4)}_4$ are
  negative.   
\end{theorem}
\begin{proof}
  We consider a general second order accurate IMEX-RK time discretisation in
  \eqref{eq:weq_rho_k_semi_disc}-\eqref{eq:weq_u_semi_disc_finupd}. Expanding 
  the unknown functions in a Taylor series, and making use of the
  conditions
  \eqref{eq:IMEX-RK-consistency}-\eqref{eq:IMEX-RK-2ndorder} for the
  IMEX-RK coefficients results in the following MPDE:  
\begin{align}
	\Dt \begin{pmatrix}
	      \rho\\
	      \mbu
	    \end{pmatrix}
	+(\bu\cdot\nabla) \begin{pmatrix}
				\rho\\
				\mbu
			      \end{pmatrix}
	+ \begin{pmatrix}
	    \br\dvg \mbu	\\
	    \frac{\ba^2}{\br\veps^2} \nabla\rho
	  \end{pmatrix}
  =\Dlt^2B^{(3)}
  \begin{pmatrix}
    \rho\\
    \mbu
  \end{pmatrix}+
  \Dlt^3B^{(4)}
  \begin{pmatrix}
    \rho\\
    \mbu
  \end{pmatrix}.   \label{eq:ieu_mpde_TD}
\end{align}
Here, the operators $B^{(3)}$ and $B^{(4)}$ contain the third and
fourth derivatives of the unknown functions $\rho$ and $\mbu$,
respectively. Since these expressions are quite long, we present them
only in the Appendix. After using the second-order order conditions
\eqref{eq:IMEX-RK-consistency}-\eqref{eq:IMEX-RK-2ndorder}, the MPDE 
\eqref{eq:ieu_mpde_TD} is free from first and second order derivatives.  

The rate of change of the energy $E$, defined in
\eqref{eq:engys_defn}, is given by       
\begin{equation} \label{Eq:ieu_EnergyNorm_TD}
  \frac{dE}{dt} = 2(U,\partial_t U) = \frac{2 \ba^2}{\br
    \veps^2} \left< \rho, \partial_t \rho \right> + 2 \br
  \langle \mbu, \partial_t \mbu \rangle. 
\end{equation}
We use the MPDE \eqref{eq:ieu_mpde_TD} in
\eqref{Eq:ieu_EnergyNorm_TD}, and integrate the resulting terms by
parts. It is easy to see that the third order derivatives do not
contribute as they all vanish in view of the periodic boundary
conditions, and hence we get
\begin{equation}
  \label{eq:ieu_rho_dtrho}
  \begin{split}
    \left< \rho, \partial_t \rho  \right> &=  \Dlt^3 \left\{b^{(4)}_1
      \frac{\ba^2}{\veps^2} \left<\rho, (\bu \cdot
        \nabla)^2 \Delta \rho \right> + b^{(4)}_2\br \left<
        \rho, (\bu \cdot\nabla)^3 (\nabla \cdot \mbu) \right>  \right. \\
    & \quad \left. + \ b^{(4)}_3\frac{\ba^2\br}{\veps^2} \left<\rho,
        (\bu \cdot \nabla) \Delta (\nabla\cdot \mbu) \right> +
      b^{(4)}_4\frac{\ba^4}{\veps^4}\left<\rho, \Delta^2 \rho \right> -
      \frac{1}{24}\left<\rho, (\bu\cdot\nabla)^4\rho\right>\right\},
  \end{split}
\end{equation}
\begin{equation}
\label{eq:ieu_u_dtu}
\begin{split}
  \left<\mbu, \partial_t \mbu  \right> & =
  \Dlt^3\left\{b^{(4)}_1\frac{\ba^2}{\veps^2} \left<\mbu, (\bu \cdot
      \nabla)^2 \nabla (\nabla \cdot \mbu) \right> +
    b^{(4)}_2\frac{\ba^2}{\br \veps^2} \left< \mbu, (\bu
      \cdot \nabla)^3 \nabla \rho \right>\right. \\  
  & \quad \left. + \ b^{(4)}_3 \frac{\ba^4}{\br
      \veps^4} \left< \mbu, (\bu \cdot \nabla) \nabla \Delta \rho \right>
    + b^{(4)}_4 \frac{\ba^4}{\veps^4} \left<\mbu, \nabla \Delta
      (\nabla \cdot \mbu)\right> - \frac{1}{24} \left<\mbu, (\bu \cdot
      \nabla)^4 \mbu\right> \right\}.
\end{split}
\end{equation}
Applying the Cauchy-Schwarz inequality in \eqref{eq:ieu_rho_dtrho} 
and \eqref{eq:ieu_u_dtu}, and using the inequalities thus obtained in
\eqref{Eq:ieu_EnergyNorm_TD}, finally yields
\begin{equation}\label{eq:ieu_energy_ineq_TD}
  \begin{split}
    \frac{dE}{dt} & \leq  2\Dlt^3 \left[ b^{(4)}_1
      \frac{\ba^4}{\br\veps^4} \norm{(\bu \cdot \nabla)
      \nabla \rho}^2 + b^{(4)}_2 \frac{\ba^2}{\veps^2} 
    \left(\norm{(\bu \cdot \nabla )^2 \rho}^2 +\norm{(\bu \cdot
        \nabla ) \nabla \cdot \mbu}^2\right) \right.\\ 
    &\quad + b^{(4)}_3 \frac{\ba^4}{\veps^4} \left(\norm{(\bu \cdot \nabla)
      \nabla \cdot \mbu}^2 + \norm{\Delta \rho}^2\right) + b^{(4)}_4
    \frac{\ba^6}{\br \veps^6} \norm{\Delta \rho}^2 -
    \frac{\ba^2}{24\br \veps^2} \norm{(\bu \cdot \nabla)^2
    \rho}^2 \\
    &\quad\left. + \ b^{(4)}_1 \frac{\br \ba^2}{\veps^2} \norm{(\bu \cdot
    \nabla) \nabla \cdot \mbu}^2 + b^{(4)}_4 \frac{\br
      \ba^4}{\veps^4} \norm{\nabla (\nabla \cdot \mbu)}^2 -
    \frac{\br}{24} \norm{(\bu \cdot \nabla )^2 \mbu}^2\right].
  \end{split}
\end{equation}
It is easy to see that the right hand side of the above inequality is
negative under the hypotheses of the theorem. 
\end{proof}
\begin{remark}
  The coefficients $b^{(4)}$ given in the Appendix are for a general
  second order IMEX-RK scheme. It can be seen that all the second
  order variants we have considered, namely ARS(2,2,2), PR(2,2,2),
  Jin(2,2,2) and RK2CN(2,2,2) are linearly $L^2$-stable by verifying
  the hypotheses of the above theorem. 
\end{remark}

\subsection{$\mcal{E}$-invariance and Asymptotic Accuracy}
\label{sec:E_inv_SD}
In this subsection, we prove the $\mcal{E}$-invariance of
the IMEX-RK time discrete scheme 
\eqref{eq:rho_ieu_sd_k}-\eqref{eq:rhou_ieu_sd_n+1} and its
asymptotic accuracy. It has been shown in \cite{dellacherie} that explicit
Godunov-type schemes in a low Mach number regime are accurate only in
one dimension, and that they are inaccurate in two and three
dimensions. Specifically, the analysis in \cite{dellacherie} shows
that there is an acoustic time scale $\tau_{ac}$ so that a
Godunov-type scheme suffers from inaccuracies beyond a time interval
of $\mcal{O}{(\tau_{ac})}$ due to the creation of spurious acoustic
waves. A cure proposed in \cite{dellacherie} is to remove the
numerical diffusion terms from the MPDE corresponding to the momentum
updates, and to use central differencing for the pressure gradient
term. However, diffusion terms are required due to stability
constraints, and deleting them may not be a feasible option always;
see also \cite{audusse} for a related discussion.

In the following, we prove the $\mcal{E}$-invariance of the
time semi-discrete scheme. However, analysing the numerical solutions
or the solutions of an MPDE resulting from a general second order IMEX
time discretisation is difficult, and hence we perform the
corresponding analysis on the linear wave equation system. The results
of linear analysis shows that the MPDE leaves the well-prepared space
$\mcal{E}$ invariant without requiring any change in the diffusion
matrices. As a result, it follows from Theorem~\ref{ThmDel} that the
solution of the MPDE are $\veps$-close to the incompressible solution,
satisfying the estimates \eqref{eq:ThmDelEst2} or
\eqref{eq:error_ap_estimate_eps} which is crucial to maintain the
asymptotic accuracy.    
\begin{theorem}
  \label{ieu_E_inv_TD}
  Suppose that at time $t^n$ the numerical solution $(\rho^n, \mbu^n)$ to
  the compressible Euler system \eqref{eq:comp} is in $\mcal{E}$, i.e.\
  $\nabla\rho^n(x)=0$ and $\dvg \mbu^n(x)=0$, for all $x$. Then at time
  $t^{n+1}$, the numerical approximation $(\rho^{n+1}, \mbu^{n+1})$
  obtained from the scheme
  \eqref{eq:rho_ieu_sd_k}-\eqref{eq:rhou_ieu_sd_n+1} satisfy
  \begin{equation}    \label{eq:sem_disc_E_inv}
    \nabla\rho^{n+1}(x)=0, \ \mbox{and} \ \dvg \mbu^{n+1}(x)=0, \
    \mbox{for all} \  x. 
  \end{equation}
  In other words, the semi-discrete scheme
  \eqref{eq:rho_ieu_sd_k}-\eqref{eq:rhou_ieu_sd_n+1}  keeps the
  well-prepared space $\mcal{E}$ invariant. 
\end{theorem} 
\begin{proof}
We use induction on the number of stages to prove the theorem. As
a first step, we prove the well-preparedness for the
first-stage, i.e.\ $k=1$, which corresponds to a fully implicit 
step. Multiplying \eqref{eq:rhou_ieu_sd_k} by $\veps^2$ and letting
$\veps \to 0$, we get
\begin{equation}
\nabla p(\rho^1) = 0 \label{eq:ieu_grad_rho_1} .
\end{equation}
Hence, the density $\rho^{1}$ at the first-stage is spatially
constant. Next, we consider the mass update \eqref{eq:rho_ieu_sd_k} and obtain 
\begin{equation}
\label{eq:ieu_mass_1}
\frac{\rho^{1} - \rho^{n}}{\Dlt} = - a_{11} \rho^{1} \nabla \cdot \mbu^{1}.
\end{equation}
We Integrate the above equation \eqref{eq:ieu_mass_1} over a domain
$\Omega$, and use the Gauss' divergence theorem to get
\begin{equation}
\label{eq:ieu_mass_1_inte}
\abs{\Omega} \frac{\rho^{1} - \rho^{n}}{\Dlt} = - a_{11}
\rho^{1} \int_{\Omega}\nabla \cdot \mbu^{1} = - a_{11}
\rho^{1} \int_{\D \Omega} \uu{\nu} \cdot \mbu^{1}.  
\end{equation}
Using periodic or wall boundary conditions, we can see that the right
most integral in the above equation \eqref{eq:ieu_mass_0_0_inte}
vanishes. As a result, $\rho^{1} =  \rho^{n}$. Since the density
$\rho^{n}$ at time $t^n$ is a constant, we immediately obtain that
$\rho^{1}$ is also a constant. Using this in equation
\eqref{eq:ieu_mass_1} yields the divergence condition   
\begin{equation}
\nabla \cdot \mbu^{1} = 0,
\end{equation} 
for the velocity $\mbu^{1}$ at the first-stage. This completes the proof
for $k = 1$.

For $k = 2$ and then on, we consider the intermediate
stages \eqref{eq:rho_ieu_sd_k}-\eqref{eq:rhou_ieu_sd_k} in the form
\eqref{eq:ieu_rho_k^_semi_disc}-\eqref{eq:ieu_rhou_k^_semi_disc}. Now
proceeding as in the case of $k = 1$, we can obtain $\rho^{k} =
\rho^{n} $. Hence, for the $k^{th}$-stage, the density$\rho^{k}$ is
same as that of $\rho^n$, and in turn we also obtain the divergence
condition:  
\begin{equation}
\nabla \cdot \mbu^{k} = 0.
\end{equation} 
As the scheme is assumed to be GSA, the numerical solution
$(\rho^{n+1}, \mbu^{n+1})$ obtained from the scheme
\eqref{eq:rho_ieu_sd_k}-\eqref{eq:rhou_ieu_sd_n+1} lives in
$\mcal{E}$, whenever the initial data $(\rho^n, \mbu^n)$ lives in
$\mcal{E}$. In other words, the IMEX-RK scheme
\eqref{eq:rho_ieu_sd_k}-\eqref{eq:rhou_ieu_sd_n+1} leaves the
well-prepared space $\mcal{E}$ invariant.  
\end{proof}

\begin{remark}
  It has to be noted that the analysis carried out in the above
  theorem doesn't take into account the correction terms introduced by
  the time discretisation. In order to substantiate the claims of the
  formal analysis thus performed in Theorem~\ref{ieu_E_inv_TD}, and to
  obtain the estimate \eqref{eq:ThmDelEst2} for the solution of the
  MPDE \eqref{eq:ieu_mpde_TD}, we prove its $\mcal{E}$-invariance in the
  following proposition. 
\end{remark}

\begin{proposition}
\label{ieu_MPDE_inv}
The modified equation system \eqref{eq:ieu_mpde_TD} for the IMEX-RK 
scheme \eqref{eq:rho_ieu_sd_k}-\eqref{eq:rhou_ieu_sd_n+1} applied to
the wave equation system is $\mcal{E}$-invariant, i.e.\ if the initial
data $(\rho(0, \uu{x}), \mbu(0, \uu{x}))$ is in $\mcal{E}$, then the
solution $(\rho(t, \uu{x}), \mbu(t, {x}))$ of the MPDE
\eqref{eq:ieu_mpde_TD} is in $\mcal{E}$ for all times $t > 0$.  
\end{proposition}
\begin{proof}
  Since the problem \eqref{eq:NondimEuler} is given on the torus
  $\mbb{T}^d$, we can assume that it is given in the whole of
  $\mbb{R}^d$ by periodic extension. It is convenient to work in the
  Fourier variables as it simplifies the calculations. Taking the
  Fourier transform of \eqref{eq:ieu_mpde_TD} in space we obtain 
  \begin{equation}
    \label{eq:ieu_mpde_TD_FT}
    \Dt \begin{pmatrix}
      \hat{\rho} \\
      \hat{\mbu}
    \end{pmatrix}
    +i(\bu\cdot\uu{\xi}) \begin{pmatrix}
      \hat{\rho} \\
      \hat{\mbu}
    \end{pmatrix}+
    i\begin{pmatrix}
      \br\uu{\xi}\cdot\hat{\mbu}\\
      \frac{\ba^2}{\br\veps^2}\uu{\xi}\hat{\rho}
    \end{pmatrix}
    = \Dlt^2\hat{B}^{(3)}\begin{pmatrix}
      \hat{\rho} \\
      \hat{\mbu}
    \end{pmatrix}
    +\Dlt^3\hat{B}^{(4)}\begin{pmatrix}
      \hat{\rho} \\
      \hat{\mbu}
    \end{pmatrix},
  \end{equation}
  where the matrices $\hat{B}^{(3)}$ and $\hat{B}^{(4)}$ are,
  respectively, the Fourier transforms of the matrices $B^{(3)}$ and
  $B^{(4)}$ introduced in Theorem~\ref{ieu_Lin_stab_an}. 
  
  Note that $(\rho,\mbu)\in\mcal{E}$ if, and only if,
  $(\hat{\rho},\hat{\mbu})\in\hat{\mcal{E}}$, where $\hat{\mcal{E}}$ is
  given by 
  \begin{equation}
    \label{eq:hat_E}
    \hat{\mcal{E}}:=\{(\hat{\rho},\hat{\mbu})\in L^2(\mbb{T}^d)^{1+d}
    \colon\hat{\rho}(\cdot,\uu{\xi})=0 \ \mbox{and} \
    \uu{\xi}\cdot\hat{\mbu}(\cdot,\uu{\xi})=0 \ \mbox{for all} \ \uu{\xi}\in\mbb{R}^d\}. 
  \end{equation}
 Hence, the quantities of interest are $\hat{\rho}$ and
 $\uu{\xi}\cdot\hat{\mbu}$. Since the first equation of
 \eqref{eq:ieu_mpde_TD_FT} is already an equation for $\hat{\rho}$,
 taking the dot product of the second equation for $\hat{\mbu}$ in
 \eqref{eq:ieu_mpde_TD_FT}  with $\uu{\xi}$ and combining with the first
 equation of \eqref{eq:ieu_mpde_TD_FT} for $\hat{\rho}$ yields a
 linear system:
 \begin{equation}
   \label{eq:ieu_FT_mpde_sys}
   \Dt \begin{pmatrix}
     \hat{\rho} \\
     \uu{\xi}\cdot\hat{\mbu}
   \end{pmatrix}
   =M(\uu{\xi}) \begin{pmatrix}
     \hat{\rho} \\
     \uu{\xi}\cdot\hat{\mbu}
   \end{pmatrix},
 \end{equation} 
 where $M(\uu{\xi})$ is a matrix consisting of polynomials in
 $\uu{\xi}$ obtained from \eqref{eq:ieu_mpde_TD_FT}. 
Let us suppose that $(\rho(0,\cdot),\mbu(0,\cdot))\in\mcal{E}$,  i.e.\
$\hat{\rho}(0,\uu{\xi})=\uu{\xi}\cdot\hat{\mbu}(0,\uu{\xi})=0$. The solution of
\eqref{eq:ieu_FT_mpde_sys} then clearly shows that
$\hat{\rho}(t,\uu{\xi})=\uu{\xi}\cdot\hat{\mbu}(t,\uu{\xi})=0$. Hence,
solution of \eqref{eq:ieu_mpde_TD} lives in $\mcal{E}$ for all times
$t>0$, if it was in $\mcal{E}$ at time $t = 0$.
\end{proof}

\begin{remark}
  Combining the above results, we can conclude that the time
  semi-discrete IMEX-RK scheme
  \eqref{eq:rho_ieu_sd_k}-\eqref{eq:rhou_ieu_sd_n+1} is AP and AA in
  the sense of Definitions~\ref{defn:AP} and \ref{defn:AA}.  
\end{remark}

\section{Analysis of the Space-time Fully-discrete Scheme} 
\label{sec:sec6}
This section is devoted to the analysis of the fully-discrete IMEX-RK
scheme resulting after a space discretisation by the finite volume
method. For the sake of simplicity, we consider only two-dimensional
(2-D) problems, and we assume that the given Cartesian spatial domain
$\Omega$ is discretised into rectangular mesh cells of lengths $\Delta
x_1$ and $\Delta x_2$ in $x_1$- and $x_2$-directions, respectively. For 
notational conveniences, let us also introduce the following finite
difference operators $\mu$ and $\delta$, defined via 
\begin{equation}	\label{eq:ieu_delta_mu_defn}
  \delta_{x_1}\omega_{i,j}:={\omega}_{i+\frac{1}{2},j}-{\omega}_{i-\frac{1}{2},j},
  \quad
  \mu_{x_1} \omega_{i,j} := \frac{\omega_{i+ \frac{1}{2}, j}+\omega_{i
      - \frac{1}{2}, j}}{2}, 
\end{equation}
in the $x_1$-direction, with analogous definitions in the
$x_2$-direction.

Given an approximation $W_{i,j}^n$ to the piecewise constant cell
averages of the conserved variable $W$ at time $t^n$, we reconstruct
all the conservative variables using a standard MUSCL-type piecewise
linear interpolant    
\begin{equation}
  P_{i,j}(x_1, x_2) = U_{i,j}^n + W_{i,j}^{\prime} (x_1 - x_{1_i}) +
  W_{i,j}^\backprime (x_2 - x_{2_j}) 
\end{equation}
in every computational cell. Here, $W_{i,j}^{\prime}$ and
$W_{i,j}^\backprime$ are, respectively, the discrete slopes in $x_1$-
and $x_2$-directions. It has to be noted that the analysis presented
in Section~\ref{sec:sec2} can be carried out to the fully-discrete
scheme, only if we restrict ourselves to smooth solutions. In such
problems, the discrete slopes $W_{i, j}^{\prime}$ and $W_{i,
  j}^{\backprime}$ are approximated using central differences, i.e.\    
\begin{equation}
  W_{i, j}^\prime := \frac{W_{i+1, j} - W_{i-1, j}}{2 \Delta x_1}, \quad
  W_{i, j}^\backprime := \frac{W_{i, j+1} - W_{i, j-1}}{2 \Delta x_2}.
\end{equation}
However, if the solution under consideration is discontinuous, we use
slope limitting techniques to obtain $W_{i, j}^{\prime}$ and $W_{i,
  j}^{\backprime}$, e.g.\
\begin{equation}
  W_{i, j}^\prime := \CWENO\left(\frac{W_{i+1, j} - W_{i, j}}{\Delta
      x_1}, \frac{W_{i, j} - W_{i-1, j}}{\Delta x_1}\right)
\end{equation}
in the $x_1$-direction. Here, the $\CWENO$ function is defined by
\begin{equation*}
  \CWENO(a,b):=\frac{\omega(a)a+\omega(b)b}{\omega(a)+\omega(b)}, \
  \omega(a):=(\delta+a^2)^{-2}, \ \delta=10^{-6}. 
\end{equation*}
Let $F_1$ and $F_2$ be, respectively, the non-stiff fluxes in $x_1$-
and $x_2$-directions, and let $G_1$ and $G_2$ be, respectively, the
stiff fluxes in $x_1$- and $x_2$-directions, cf.\
\eqref{eq:flux_split}. The eigenvalues of the Jacobians of $F_1(W)$
and $F_2(W)$ with respect to $W$ can be obtained as $\ld_{1,1}=0,
\ld_{1,2}=u_1, \ld_{1,3}=2 u_1$, and $\ld_{2,1}=0, \ld_{2,2}=u_2,
\ld_{2,3}=2 u_2$. The timestep $\Dlt$ at time $t^n$ is computed by the
CFL condition:
\begin{equation}
  \label{eq:CFL}
  \Dlt\max_{i,j}\max\left(\frac{\left|\ld_{1,3}\left(W_{i,j}^n\right)\right|}{\Delta
      x_1}, \frac{\left|\ld_{2,3}\left(W_{i,j}^n\right)\right|}{\Delta
      x_2}\right)=\nu, 
\end{equation} 
where $\nu<1$ is the given CFL number. Note that the above condition
is almost like the advective CFL condition, and is independent of
$\veps$.  

\subsection{Space-time Fully-discrete Scheme}
Discretising the semi-discrete scheme
\eqref{eq:rho_ieu_sd_k}-\eqref{eq:rhou_ieu_sd_n+1} using the finite
volume method, and replacing the physical fluxes by suitable numerical
fluxes yields the fully-discrete scheme. In order to introduce the
numerical fluxes, we first rewrite the IMEX time semi-discrete scheme
\eqref{eq:rho_ieu_sd_k}-\eqref{eq:rhou_ieu_sd_n+1} in terms of the
stiff and non-stiff flux functions, via
\begin{align}
  W^k &= W^n - \Dlt \tilde{a}_{k,\ell} \partial_{x_m} F_{m}(W^\ell) - \Dlt
        a_{k,l} \partial_{x_m} G_{m}(W^l), \ k = 1,2, \ldots,
        s, \label{eq:ieu_sd_k} \\
  W^{n+1} &= W^n - \Dlt \tilde{\omega}_{k} \partial_{x_m} F_{m}(W^k) -
            \Dlt \omega_{k} \partial_{x_m} G_{m}(W^k). \label{eq:ieu_sd_n+1}
\end{align}  
\begin{definition} 
The $k^{th}$ stage of an $s$-stage fully-discrete IMEX-RK scheme for the
Euler system \eqref{eq:comp} is defined as
\begin{equation}
\label{eq:ieu_FD_k}
W^k_{i,j} = W^n_{i,j} - \tilde{a}_{k, \ell} \nu_m \delta_{x_m}
\mcal{F}_{m,i,j}^\ell- a_{k, l} \nu_m \delta_{x_m}
\mcal{G}_{m,i,j}^l, \  k = 1, 2,\ldots, s 
\end{equation}
with the numerical solution at time $t^{n+1}$ given by   
\begin{equation}
\label{eq:ieu_FD_n+1}
W^{n+1}_{i,j} = W^n_{i,j} - \tilde{\omega}_{k}\nu_m \delta_{x_m}
\mcal{F}_{m,i,j}^k- \omega_{k} \nu_m \delta_{x_m}
\mcal{G}_{m,i,j}^k,
\end{equation}
where the repeated index $m$ takes values in $\{1,2\}$,
$\nu_m:=\frac{\Delta t}{\Delta x_m}$ denote the mesh ratios, and
$\mcal{F}_{m}$ and $\mcal{G}_{m}$ are, respectively, the numerical
fluxes used to approximate the physical fluxes $F_{m}$ and $G_{m}$. In
this paper, we have used a simple Rusanov-type flux for
$\mcal{F}_{m}$ and central flux for $\mcal{G}_{m}$, defined, e.g.\
along $x_1$- and $x_2$-directions as  
\begin{equation}
\begin{aligned}
\mcal{F}_{1, i+\frac{1}{2}, j}^k &= \frac{1}{2} \left(F_1\left(W_{i+\frac{1}{2},j}^{k+}\right) +
  F_1\left(W_{i+\frac{1}{2},j}^{k-}\right) \right) - \frac{\alpha_{1, i + \frac{1}{2},j}}{2}
\left(W_{i+\frac{1}{2},j}^{k+} - W_{i+\frac{1}{2}, j}^{k-}\right),\\
\mcal{F}_{2, i, j+\frac{1}{2}}^k &= \frac{1}{2} \left(F_2\left(W_{i,j+\frac{1}{2}}^{k+}\right) +
  F_2\left(W_{i, j+\frac{1}{2}}^{k-}\right) \right) - \frac{\alpha_{2, i,j + \frac{1}{2}}}{2}
\left(W_{i,j+\frac{1}{2}}^{k+} - W_{i, j+\frac{1}{2}}^{k-}\right),\\
\mcal{G}_{1, i+ \frac{1}{2}, j} &= \frac{1}{2} \left( G_1( W_{i+1, j}^{k}) +
  G_1(W_{i, j}^{k}) \right), \quad
\mcal{G}_{2, i, j+ \frac{1}{2}} = \frac{1}{2} \left(G_2(W_{i, j+1}^{k}) +
  G_2(W_{i, j}^{k}) \right).
\end{aligned}
\end{equation}
Here, $W^{k\pm}_{i+\frac{1}{2}, j}$ and
$W^{k\pm}_{i,j+\frac{1}{2}}$ are the interpolated states obtained
using the piecewise linear reconstructions, and the wave-speeds are
computed as, e.g.\ in the $x_1$-direction
\begin{equation}
  \label{eq:wave_speed}
  \alpha_{1, i +
    \frac{1}{2},j}:=\max\left(\left|\ld_{1,3}\left(W_{i+\frac{1}{2},j}^{k-}\right)\right|,
      \left|\ld_{1,3}\left(W_{i+\frac{1}{2},j}^{k+}\right)\right|\right). 
\end{equation}
\end{definition}

The rest of this section is devoted to establishing the AP and AA
properties of the fully-discrete scheme
\eqref{eq:ieu_FD_k}-\eqref{eq:ieu_FD_n+1} as done in
Section~\ref{sec:sec5} for the case of the time semi-discrete scheme. 

\subsection{AP Property}

\subsubsection{Asymptotic Consistency}
\label{sec:asympt-cons-1}

\begin{theorem} \label{Asym_cons_FD}
Suppose that the data at time $t^{n}$ are well-prepared, i.e.\
$\rho^{n}$ and $\mbu^{n}$ admit the decomposition:
\begin{align}
  \rho_{i,j} &= \rho_{(0),i,j} + \veps^2 \rho_{(2),i,j}
             , \label{eq:rho_wp_fully-disc} \\  
  \mbu_{i,j} &= \mbu_{(0),i,j} + \veps
                \mbu_{(1),i,j}, \label{eq:u_wp_fully-disc}  
\end{align}
where $\hat{\nabla} \rho^{n}_{(0),i,j} = 0$ and $\hat{\nabla} \cdot
\mbu^{n}_{(0),i,j} = 0$. Here, $\hat{\nabla}$ is the discrete gradient
introduced by the implicit terms, i.e.\ by replacing the derivatives
by central differences. Then for a GSA scheme, if  the intermediate
solutions $\rho^{k}_{i,j}$ and $\mbu^{k}_{i,j}$ defined by
\eqref{eq:ieu_FD_k} admit the same decomposition
\eqref{eq:rho_wp_fully-disc}-\eqref{eq:u_wp_fully-disc}, they
satisfy  $\hat{\nabla} \rho^{k}_{(0),i,j} = 0$ and $\hat{\nabla} \cdot
\mbu^{k}_{(0),i,j} = 0$, i.e.\ they are well-prepared as well. As a
consequence, if the numerical solutions $\rho^{n+1}_{i,j}$ and
$\mbu^{n+1}_{i,j}$ admit the decomposition
\eqref{eq:rho_wp_fully-disc}-\eqref{eq:u_wp_fully-disc}, then they
are also well-prepared. In other words, the fully discrete scheme
\eqref{eq:ieu_FD_k}-\eqref{eq:ieu_FD_n+1} is asymptotically 
consistent in the sense of Definition~\ref{defn:AP}.    
\end{theorem} 
\begin{proof}
We follow similar lines as in the proof of Theorem~\ref{Asym_cons_SD}
to establish the asymptotic consistency of the fully-discrete
scheme. We plugin the ansatz \eqref{eq:f_ansatz} in a discrete form 
for all the dependent variables in
\eqref{eq:ieu_FD_k}-\eqref{eq:ieu_FD_n+1}. Equating to zero the
$O(\veps^{-2})$ terms in the first stage, i.e.\ $k=1$, which is a fully
implicit step, gives  
\begin{equation}
\left(\frac{\mu_{x_1} \delta_{x_1}}{\Delta x_1} p \left( \rho^{1}_{(0), i, j}
    \right),
  \frac{\mu_{x_2} \delta_{x_2}}{\Delta x_2} p \left( \rho^{1}_{(0), i, j}
    \right)\right) =
0 \label{eq:ieu_grad_rho_1_0_FD} . 
\end{equation}
Hence, the zeroth order density $\rho^{1}_{(0), i, j}$ is constant for
all $i,j$, if we assume periodic or wall boundary conditions. In an
analogous way, we can show that the first order density
$\rho^{1}_{(1), i, j}$ is constant for all $i,j$.  

Next, we consider  the $\mcal{O}(1)$ terms in the mass update of
\eqref{eq:ieu_FD_k} to obtain  
\begin{equation}
\label{eq:ieu_mass_0_0_FD}
\frac{\rho^{1}_{(0), i, j} - \rho^{n}_{(0), i, j}}{\Dlt} = - a_{11}
\rho^{1}_{(0), i,j} \frac{\mu_{x_m} \delta_{x_m}}{\Delta x_m}u^{1}_{m,(0), i,j}.
\end{equation}

If the boundary conditions of the problem are periodic or wall, it can
easily be seen that the right hand side of equation
\eqref{eq:ieu_mass_0_0_FD}, when summed over the entire domain,
vanishes to yield
\begin{equation}
\label{eq:ieu_mass_0_0_inte_disc}
\frac{\rho^{1}_{(0), i, j} - \rho^{n}_{(0), i, j}}{\Dlt} =  0.
\end{equation}
As a result, $\rho^{1}_{(0), i,j} =  \rho^{n}_{(0),i,j}$, for all
$i,j$. Since the zeroth order density $\rho^{n}_{(0)}$ is a constant,
$\rho^{1}_{(0),i,j}$  as well as $\rho^{1}_{(1),i,j}$ are
constants. Using this in equation \eqref{eq:ieu_mass_0_0_FD}, we
obtain the discrete divergence condition: 
\begin{equation}
\frac{\mu_{x_m} \delta_{x_m}}{\Delta x_m}u^{1}_{m,(0), i,j} = 0 
\end{equation} 
for the leading order velocity $\mbu^{1}_{(0)}$, i.e.\ the discrete
divergence vanishes for all $i,j $. This completes the proof
for $k = 1$. 

For $k = 2$ and then on, we can mimic the steps carried out in the
proof of Theorem~\ref{Asym_cons_SD} for each of the intermediate
stages, and conclude that the discrete gradient of the density
$\rho^k_{(0)}$, and the discrete divergence of the velocity
$\mbu^k_{(0)}$ vanish. Finally, we can prove that the numerical
solution at time $t^{n+1}$ is also well-prepared along similar lines
as in the proof of Theorem~\ref{Asym_cons_SD}. 



Summarising the above steps, we can clearly see that the limt of the
scheme \eqref{eq:ieu_FD_n+1} is a consistent discretisation of the
incompressible limit system \eqref{eq:incomp}. Hence, the fully-discrete
scheme \eqref{eq:ieu_FD_k}-\eqref{eq:ieu_FD_n+1} is asymptotically
consistent.   
\end{proof}
\subsubsection{Linearised $L^2$-stability Analysis}
We now proceed to establish the linear $L^2$-stability of the
fully-discrete IMEX-RK scheme
\eqref{eq:ieu_FD_k}-\eqref{eq:ieu_FD_n+1} by applying it to the linear
wave equation system. Analogous to Subsection~\ref{sec:line-l2-stab},
we obtain a sufficient condition for stability which involves only the
RK coefficients. The derivation of MPDE corresponding to a space-time 
fully-discrete general second order IMEX scheme turns out to be
extremely complicated and hence, we restrict ourselves only to a first 
order IMEX discretisation; see \cite{pareschi-russo-conf} for an
example of a first order IMEX scheme other than Euler(1,1,1).     
\begin{theorem}
  \label{ieu_Lin_stab_anF}
  A first order space-time fully-discrete IMEX-RK scheme given by
  \eqref{eq:ieu_FD_k}-\eqref{eq:ieu_FD_n+1} is linearly $L^2$-stable
  under a CFL-like condition $\Delta t=C\min(\Delta x_1,\Delta x_2)$,
  where $C$ is a constant independent of $\veps$, if
  \begin{equation}
    \label{eq:b12b32}
    b_1^{(2)} < 0, \ \mbox{and} \ b_3^{(2)} > 0,
  \end{equation}
  where the coefficients $b_1^{(2)}$ and $b_3^{(2)}$ are defined by
  \eqref{eq:def_b_4} in Appendix.   
\end{theorem}
\begin{proof}
The MPDE for the IMEX-RK scheme
\eqref{eq:ieu_FD_k}-\eqref{eq:ieu_FD_n+1} can be obtained as 
\begin{align}
  \Dt \begin{pmatrix}
    \rho\\
    \mbu
  \end{pmatrix}
  +(\bu\cdot\nabla) \begin{pmatrix}
    \rho\\
    \mbu
  \end{pmatrix}
  + \begin{pmatrix}
    \br\dvg \mbu	\\
    \frac{\ba^2}{\br\veps^2} \nabla\rho
  \end{pmatrix}
  =\mcal{B}^{(2)} \begin{pmatrix}
    \rho\\
    \mbu
  \end{pmatrix},    \label{eq:ieu_mpde_FD}
\end{align}
where the operator $\mcal{B}^{(2)}$ is defined in Appendix. As done
in the case of the time semi-discrete scheme in
Theorem~\ref{ieu_Lin_stab_an}, we use the MPDE \eqref{eq:ieu_mpde_FD}
in the equation \eqref{Eq:ieu_EnergyNorm_TD} to calculate the rate of
change of energy. The right hand side of \eqref{Eq:ieu_EnergyNorm_TD}
then gives
\begin{equation}
  \label{eq:ieu_rho_dtrho_FD} 
  \begin{aligned}
    \left< \rho, \partial_t \rho  \right> &= \Dlt \left\{ b^{(2)}_1 \left< \rho,
        (\bu \cdot \nabla)^2
        \rho \right>  +  b^{(2)}_2
      \br \left< \rho,
        (\bu \cdot \nabla )
        \nabla \cdot \mbu\right> + b^{(2)}_3
      \frac{\ba^2}{\veps^2} \left< \rho, \Delta \rho \right>
    \right\} \\
    & \quad + \frac{1}{2} \left< \rho, \Delta x_1 \abs{{\baru}_1} \partial_{x_1}^2
      \rho + \Delta x_2 \abs{{\baru}_2} \partial_{x_2}^2
      \rho \right>,
  \end{aligned}
\end{equation}
\begin{equation}
  \label{eq:ieu_u_dtu_FD}
  \begin{aligned} 
    \left< \mbu, \partial_t \mbu  \right> &= \Dlt \left\{ b^{(2)}_1 \left< \mbu,
      (\bu \cdot \nabla)^2 \mbu \right>  +  b^{(2)}_2
    \frac{\ba^2}{\br \veps^2} \left< \mbu, (\bu \cdot \nabla )
      \nabla \rho\right> + 
    b^{(2)}_3 \frac{\ba^2}{\veps^2} \left< \mbu, \nabla (\nabla \cdot \mbu) \right>
    \right\}  \\ 
    &\quad+ \frac{1}{2} \left< \mbu, \Delta x_1 \abs{{\baru}_1} \partial_{x_1}^2
      \mbu + \Delta x_2 \abs{{\baru}_2} \partial_{x_2}^2 \mbu \right>.
  \end{aligned}
\end{equation}
Regrouping the terms in the above equations after applying the
Cauchy-Schwarz and AM-GM inequalities, and then using the inequalities
thus obtained in \eqref{Eq:ieu_EnergyNorm_TD}, finally yields
\begin{equation}
  \label{eq:ieu_energy_ineq_FD}
  \begin{aligned}
    \frac{dE}{dt} \leq & -\Dlt b^{(2)}_3 \left\{ \frac{2
      \ba^4}{\br\veps^4} \norm{\nabla \rho}^2 
    +  \frac{2 \ba^2 \br}{\veps^2} \norm{\nabla
    \cdot \mbu}^2 \right\}  - \left\{\left(\Delta x_1
    \frac{\abs{{\baru}_1}}{2}  + \Dlt \norm{\bu}^2 b^{(2)}_1\right)
  \left( \frac{2 \ba^2}{\br \veps^2}
    \norm{\partial_{x_1} \rho}^2  \right. \right.\\
  & \quad \left.\left.+ 2 \br \norm{\partial_{x_1} \mbu}^2\right) + 
    \left(\Delta x_2 \frac{\abs{{\baru}_2}}{2}  + \Dlt \norm{\bu}^2
      b^{(2)}_1 \right) \left( \frac{2 \ba^2}{\br \veps^2}
      \norm{\partial_{x_2} \rho}^2 + 2 \br \norm{\partial_{x_2}
      \mbu}^2\right)\right\}. 
  \end{aligned}
\end{equation}
If the coefficients $-b^{(2)}_1$ and $b^{(2)}_3$ are positive, it can
easily be seen that the overall expression on the right hand side of
the above inequality is negative under the following CFL-like
condition:
\begin{equation}
\Dlt \leq -b^{(2)}_1\frac{\min({\baru}_1, {\baru}_2)}{\norm{{\bu}}^2}
\min(\Delta x_1, \Delta x_2).
\label{eq:CFL_first_order}
\end{equation}
\end{proof}
\begin{remark}
  Note that $-b^{(2)}_1$ and $b^{(2)}_3$ are positive for
  Euler(1,1,1). An analogous $L^2$-stability analysis is carried out
  in \cite{dimarco_loubere_siamjsc_2017} and a condition similar to
  \eqref{eq:CFL_first_order} is obtained. Further, the analysis in
  \cite{dimarco_loubere_siamjsc_2017}  also show that more stringent
  conditions are to be enforced to obtain $L^\infty$-stability.  
\end{remark}
\subsection{$\mcal{E}$-invariance and Asymptotic Accuracy}
As done in Section~\ref{sec:sec5} for the time semi-discrete scheme,
we first formally prove that the fully-discrete scheme
\eqref{eq:ieu_FD_k}-\eqref{eq:ieu_FD_n+1} leaves $\mcal{E}$
invariant. To obtain the estimate \eqref{eq:ThmDelEst2}, we also carry
out an analogous MPDE analysis of the linearised scheme. Since the
proofs of the following results follow similar lines as that of their
semi-discrete counterparts, we do not intent to present the details.  
\begin{theorem}
 Suppose that at time $t^n$ the numerical solution $(\rho^n_{i,j},
 \mbu^n_{i,j})$ to the compressible Euler system \eqref{eq:comp} is in
 $\mcal{E}$, i.e.\ $\rho^n_{i,j}=\const$ and $\hat{\nabla}\cdot
 \mbu^n_{i,j}=0$ for all $i,j$. Then, at time $t^{n+1}$, the numerical
 approximation $(\rho^{n+1}_{i,j}, \mbu^{n+1}_{i,j})$ obtained from the
 scheme \eqref{eq:ieu_FD_k}-\eqref{eq:ieu_FD_n+1} satisfy     
 \begin{equation}    \label{eq:sem_disc_E_inv1}
	\rho^{n+1}_{i,j}=\const, \quad \frac{\mu_{x_m}
          \lambda_{x_m}}{\Delta x_m} u^{n+1}_{m, i, j}=0, \ \mbox{for all} \
        i,j.
      \end{equation}
In other words, the fully-discrete scheme
\eqref{eq:ieu_FD_k}-\eqref{eq:ieu_FD_n+1} keeps the well-prepared
space $\mcal{E}$ invariant. 
\end{theorem} 
\begin{proposition}
  \label{ieu_MPDE_inv_FD}
  The modified equation system \eqref{eq:ieu_mpde_FD} for a
  first order IMEX-RK scheme
  \eqref{eq:ieu_FD_k}-\eqref{eq:ieu_FD_n+1}, applied to the linear wave
  equation system \eqref{eq:WaveHL}, is $\mcal{E}$-invariant, i.e.\ if
  the initial data $(\rho(0, \uu{x}), \mbu(0, \uu{x}))$ is in $\mcal{E}$, then the
  solution $(\rho(t, \uu{x}), \mbu(t, \uu{x}))$ of the MPDE
  \eqref{eq:ieu_mpde_FD} is in $\mcal{E}$ for all times $t > 0$.
\end{proposition}

\section{Numerical Case Studies}
\label{sec:sec7}

In this section, we present the results of our numerical experiments in
order to substantiate the claims made in the previous sections.
In the first test problem, we consider an advecting vortex for which a
smooth exact solution is available. We make use of the exact solution
to test the experimental order of convergence (EOC) for different
$\veps$. The results obtained clearly demonstrate the uniform second
order convergence with respect to $\veps$. In addition, the results also
show that the dissipation of the scheme is independent of $\veps$. In
the second test problem, we start with an exact analytical solution of
the incompressible equations, and use this solution to measure the
convergence of the numerical solution to the incompressible
solution. The convergence study clearly yields second order
convergence for very small values of $\veps$, or the stiff accuracy in
the incompressible regime. In the last problem, we consider a
discontinuous solution corresponding to a cylindrical explosion
therein we set $\veps=1$. The scheme captures the shock wave and gives
a good performance in the compressible regime as well.  
  
\subsection{Experimental Order of Convergence}
\label{sec:problem-1}
Drawing inspiration from the traveling vortex problem studied in
\cite{yelash}, we appropriate the initial conditions for the isentropic
Euler system as follows.  
 \begin{equation*}
  \begin{aligned}
    \rho(0,\uu{x}) &= 1.9+ \left(\frac{\Gamma \eta }{\omega} \right)^2 \left(
      k( \omega r ) - k (\pi) \right) \chi_{\omega r \leq \pi} , \\
    u_1(0, \uu{x}) &=0.6 + \Gamma ( 1 + \cos (\omega r)) (0.5 - x_2)
    \chi_{\omega r \leq \pi}, \\
    u_2 (0, \uu{x}) &= \Gamma ( 1 + \cos (\omega r)) (x _1 - 0.5)
    \chi_{\omega r \leq \pi},
  \end{aligned}
\end{equation*}
where $r = \norm{\uu{x} - (0.5, 0.5)}$, $\Gamma = 1.5$, $\omega = 4 \pi$,
and $k(r) = 2 \cos r + 2r \sin r + \frac{1}{8}\cos(2r) +
\frac{1}{4}r\sin(2r) + \frac{3}{4} r^2$. Here, $\Gamma$ is a parameter
known as the vortex intensity, $r$ denotes the distance from the core
of the vortex, and $\omega$ is an angular wave frequency specifying
the width of the vortex. The Mach number $\veps$ is controlled by
adjusting the value of $\eta$ via the relation $\veps = {0.6\eta}/{\sqrt{110}}$. 

The above problem admits an exact solution 
$(\rho,u_1,u_2)(t,x_1,x_2)=(\rho,u_1,u_2)(0,x_1-0.6t,x_2)$. The
computational domain $\Omega = [0,1] \times [0,1]$ is
successively divided to $10\times 10,20\times 20$ up to $80\times
80$ square mesh cells. All the four boundaries are set to be
periodic. The computations are performed up to a time $T=0.1$ using
the second order ARS(2,2,2) scheme. The CFL number is set as
$0.45$. The EOC computed in $L^1$ and $L^2$ norms using the above
exact solution, for $\veps$ ranging in
$\{10^{-6},10^{-5},\ldots,10^{-1}\}$, are given in
Table~\ref{tab:eocepsm}. The table clearly shows
uniform second order convergence of the scheme with respect to
$\veps$.   
\begin{table}[htbp]
  \centering
  \begin{tabular}[htbp]{|c|c|c|c|c|c|c|c|c|}
    \hline
    $N$ & $L^1$ error in $u_1$ & EOC & $L^2$  error in $u_1$ & EOC &
                                                                     $L^1$
                                                                     error
                                                                     in
                                                                     $u_2$
    & EOC& $L^2$ error in $u_2$ & EOC \\ 
    \hline 

10 & 4.9120e-03 & &8.8088e-03 &  &7.9419e-03& &1.7131e-02 & \\
\hline
20 & 1.3454e-03 &	1.8682& 2.5016e-03 & 1.8160 &
                                                                   2.6670e-03
    & 1.5742 &  5.7609e-03 & 1.5722 \\
\hline
40 & 3.1818e-04 &	2.0800 & 6.5165e-04 & 1.9406 &
                                                                     6.5744e-04
    & 2.0203 & 1.4305e-03 & 2.0097\\
\hline
80 & 8.0467e-05 & 1.9834 & 1.8591e-04 & 1.8094 &
                                                                   1.6067e-04
    & 2.0327 &3.6866e-04 & 1.9561 \\
\hline
\hline
10 & 4.9125e-03 & & 8.8099e-03  & & 7.9430e-03 & &
                                                             1.7131e-02 & \\
\hline
20 & 1.3445e-03 & 1.8693 & 2.5017e-03 & 1.8162 &
                                                         2.6657e-03
    & 1.5751 & 5.7604e-03 & 1.5724 \\
\hline
40 & 3.1752e-04 & 2.0822 & 6.5147e-04 & 1.9411 &
                                                         6.5705e-04
    & 2.0205 & 1.4302e-03 & 2.0099\\
\hline
80 & 7.7912e-05 &	2.0270 & 1.8559e-04 & 1.8116&
                                                          1.5926e-04
    & 2.0446 & 3.6858e-04 &1.9562\\ 
    \hline
    \hline 
    10 & 4.9125e-03&  & 8.8099e-03  & & 7.9430e-03&
          &1.7131e-02 & \\
\hline
20 & 1.3445e-03& 1.8693 & 2.5017e-03 & 1.8162 &
                                                        2.6657e-03
    & 1.5751  & 5.7604e-03 & 1.5724 \\
\hline
40 & 3.1752e-04 & 2.0822 &	6.5147e-04 & 1.9411 &
                                                          6.5705e-04
    & 2.0205 & 1.4302e-03 & 2.0099\\
\hline
80 & 7.7911e-05 &  2.0270 & 1.8558e-04& 1.8116 &
                                                         1.5926e-04
    &  2.0446 & 3.6858e-04 &	1.9562 \\
    \hline
    \hline
10 & 4.9125e-03 & &8.8099e-03 & & 7.9430e-03 & & 1.7131e-02 &  \\
\hline
20 & 1.3445e-03 & 1.8693 &2.5016e-03 & 1.8162 & 2.6657e-03 & 1.5751 &
                                                                      5.7604e-03 &1.5724 \\ 
\hline
40 & 3.1740e-04 & 2.0827 & 6.5149e-04 & 1.9411& 6.5702e-04
    &2.0205 &1.4302e-03 &2.0099 \\ 
\hline
80 & 7.7824e-05 &	2.0280 & 1.8558e-04 & 1.8117 &
                                                           1.5923e-04
    & 2.0448 & 3.6854e-04 & 1.9563\\  
    \hline
    \hline 
10 & 4.9117e-03 &	&8.8096e-03 & &  7.9411e-03 &
          &1.7130e-02 & \\
\hline
20 & 1.3433e-03 & 1.8704 & 2.5006e-03 &1.8167 & 2.6647e-03
    & 1.5753 & 5.7578e-03 & 1.5729\\
\hline
40 & 3.1573e-04 & 2.0890 & 6.5096e-04 &1.9416 &
                                                                  6.5528e-04
    & 2.0237 & 1.4272e-03 & 2.0123 \\ 
\hline
80 & 8.1065e-05 & 1.9615 &  1.8943e-04 & 1.7808 &
1.5865e-04 & 2.0462 & 3.6843e-04 &1.9537
    \\
    \hline
    \hline
10 & 4.9009e-03& & 8.7985e-03& & 7.9211e-03 & &1.7111e-02&  \\
\hline			
20 & 1.3395e-03 &	1.8713 & 2.5124e-03 & 1.8081 & 2.6297e-03
    & 1.5907 & 5.7013e-03 &1.5856\\
\hline
40 & 3.4569e-04 & 1.9541 & 7.1801e-04 & 1.8070 & 6.4199e-04
    & 2.0342 & 1.4288e-03 & 1.9964 \\ 
\hline 
80 & 1.1985e-04 & 1.5282 & 2.6469e-04 &1.4396 & 1.7083e-04&
                                                                 1.9099& 3.9443e-04 & 1.8570\\
\hline
  \end{tabular}
  \caption{$L^1$ and $L^2$ errors in $u_1$ and $u_2$ and the EOC 
    computed for Problem~\ref{sec:problem-1}. Top to bottom:
    $\veps=10^{-6}$ to $\veps=10^{-1}$.}
  \label{tab:eocepsm}
\end{table}

Next, we give in Figure~\ref{fig:inimach}, the pcolor plot of the
initial Mach number distribution at time $T=0$. In
Figure~\ref{fig:mach}, the pcolor plots of physical Mach numbers at time
$T=1.67$, which is the taken by the vortex to complete one cycle, for
the above range of values of $\veps$ is given. Comparing each of the
Mach profiles in Figure~\ref{fig:mach} with the initial one in
Figure~\ref{fig:inimach}, it is very evident that the vortex does not
deform or degrade much as it advects. Visually, one cannot observe any
differences among the plots in Figure~\ref{fig:mach}, which confirms
that the dissipation of the vortex is independent of $\veps$. This
fact is confirmed also by the behaviour of the relative kinetic
energy, and the vorticity for different $\veps$. In
Figure~\ref{fig:ke_vort}, the distribution of the relative kinetic
energy as a function of time from $T=0$ to $T=1.67$, and the cross
section of the vorticity at $x_2=0$ and $T=1.67$ as a function of
$x_1$ is presented. The plots show that the relative kinetic energy is
almost equal to one, independent of $\veps$. It is quite remarkable
that even for $\veps=10^{-6}$, the energy dissipation is only
$0.03\%$. Similarly, the vorticity plot also confirms that the
dissipation is almost independent of $\veps$.     

\begin{figure}[htbp]
  \centering
  \includegraphics[ scale = 0.45, trim={1cm 0 1.7cm 0},
  clip]{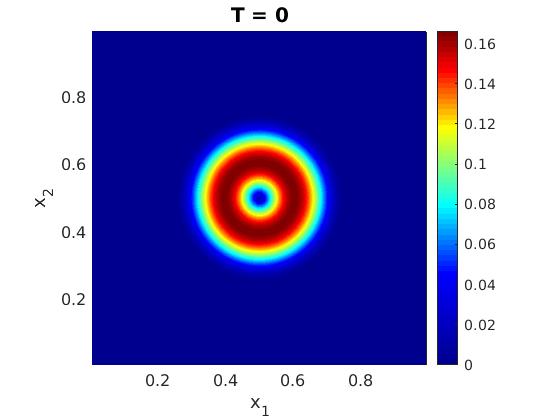}
\caption{Pcolor plots of the initial Mach number profile for the vortex
   problem.} 
  \label{fig:inimach}
\end{figure}
\begin{figure}[htbp]
  \centering
  \includegraphics[scale = 0.36, trim={1cm 0 3.7cm 0},
  clip]{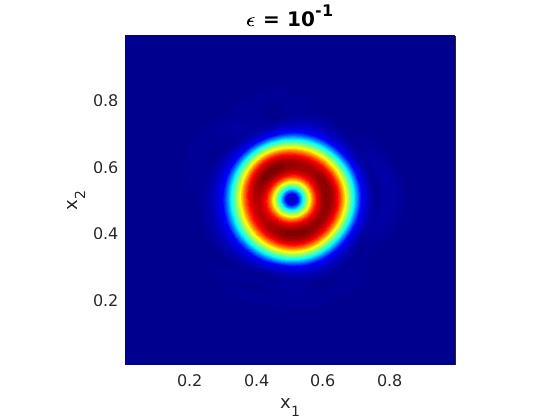} 
\hspace{-2mm}
\includegraphics[ scale = 0.36, trim={4.2cm 0 3.5cm 0},
clip]{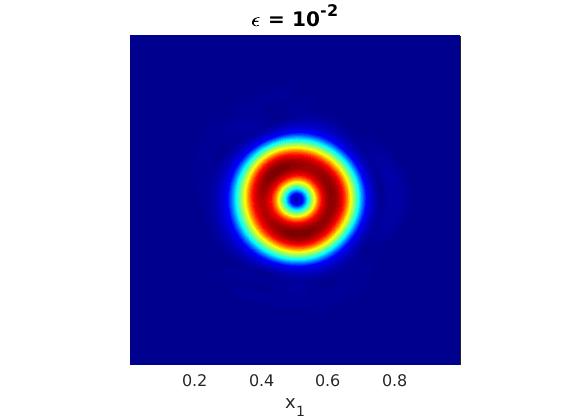}
\hspace{-2mm}
\includegraphics[ scale = 0.36, trim={4cm 0 1.7cm 0},
clip]{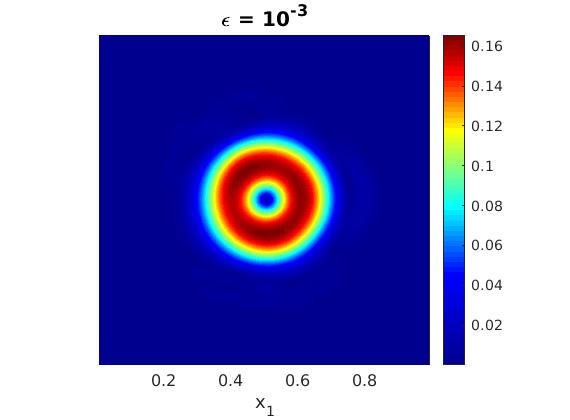}
\includegraphics[ scale = 0.36, trim={1cm 0 3.6cm 0},
  clip]{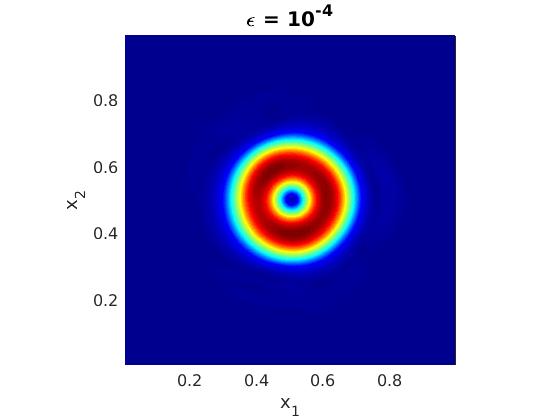}
\hspace{-2mm}
\includegraphics[ scale = 0.36, trim={4.1cm 0 3.4cm 0},
clip]{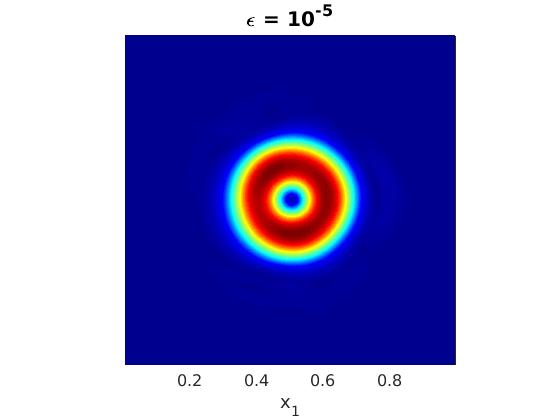}
\hspace{-2mm}
\includegraphics[ scale = 0.36, trim={3.9cm 0 1.7cm 0},
clip]{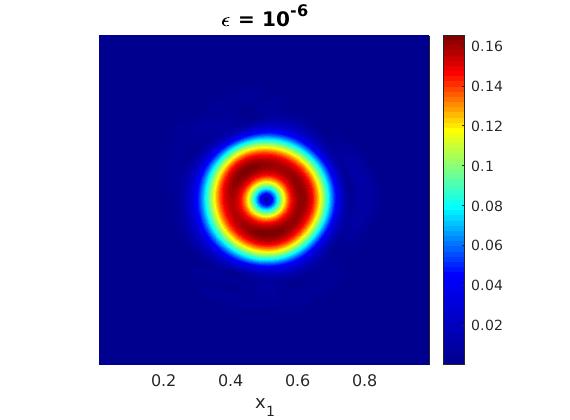}
  \caption{Pcolor plots of the Mach number profiles for the
    vortex problem at $T = 1.67$ for  $\veps \in \{ 10^{-1}, 10^{-2},  10^{-3},
   10^{-4},10^{-5}, 10^{-6}\}$.} 
  \label{fig:mach}
\end{figure}

\begin{figure}[htbp]
  \centering
  \includegraphics[height=0.23\textheight]{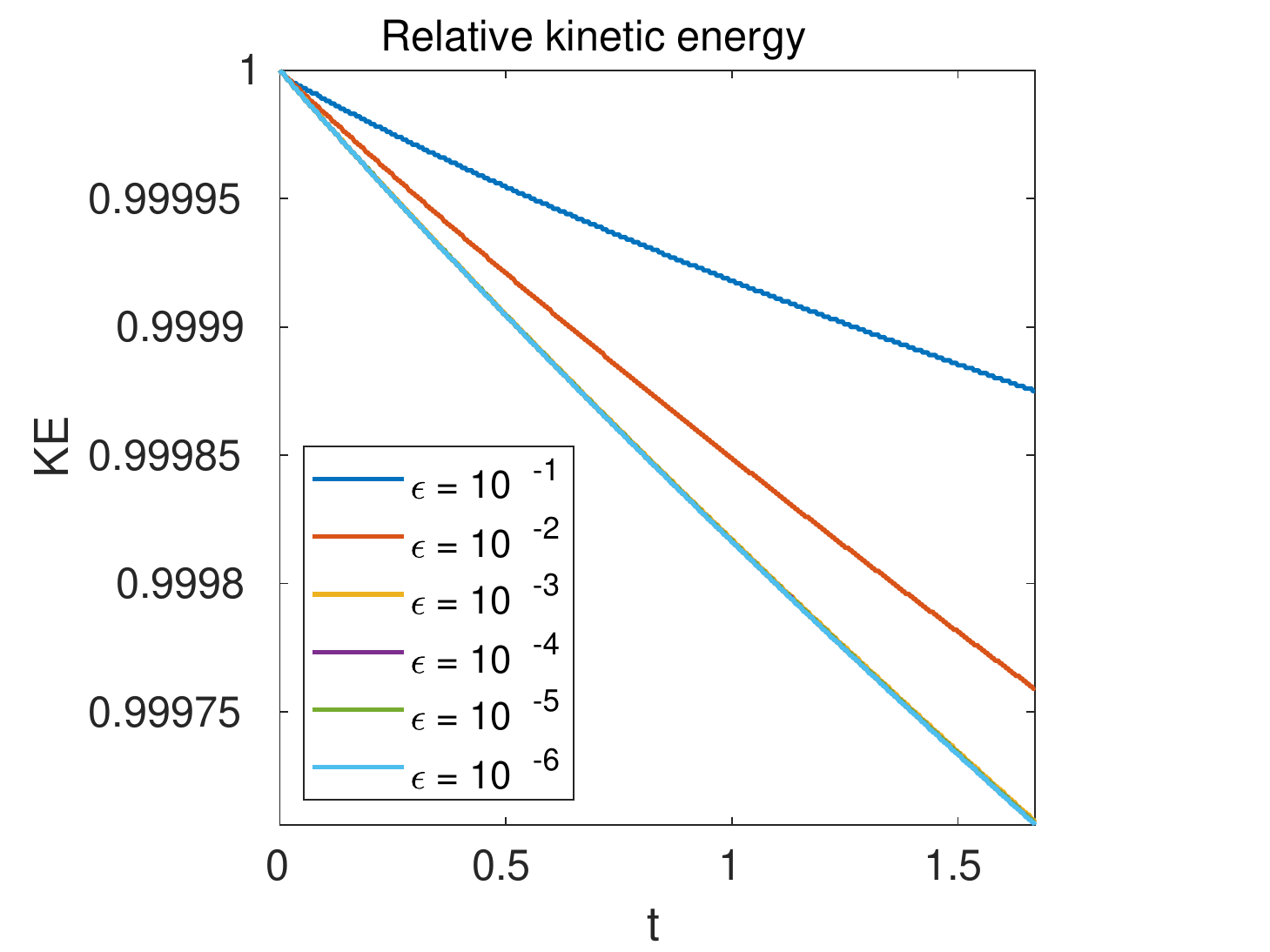}
  \includegraphics[height=0.23\textheight]{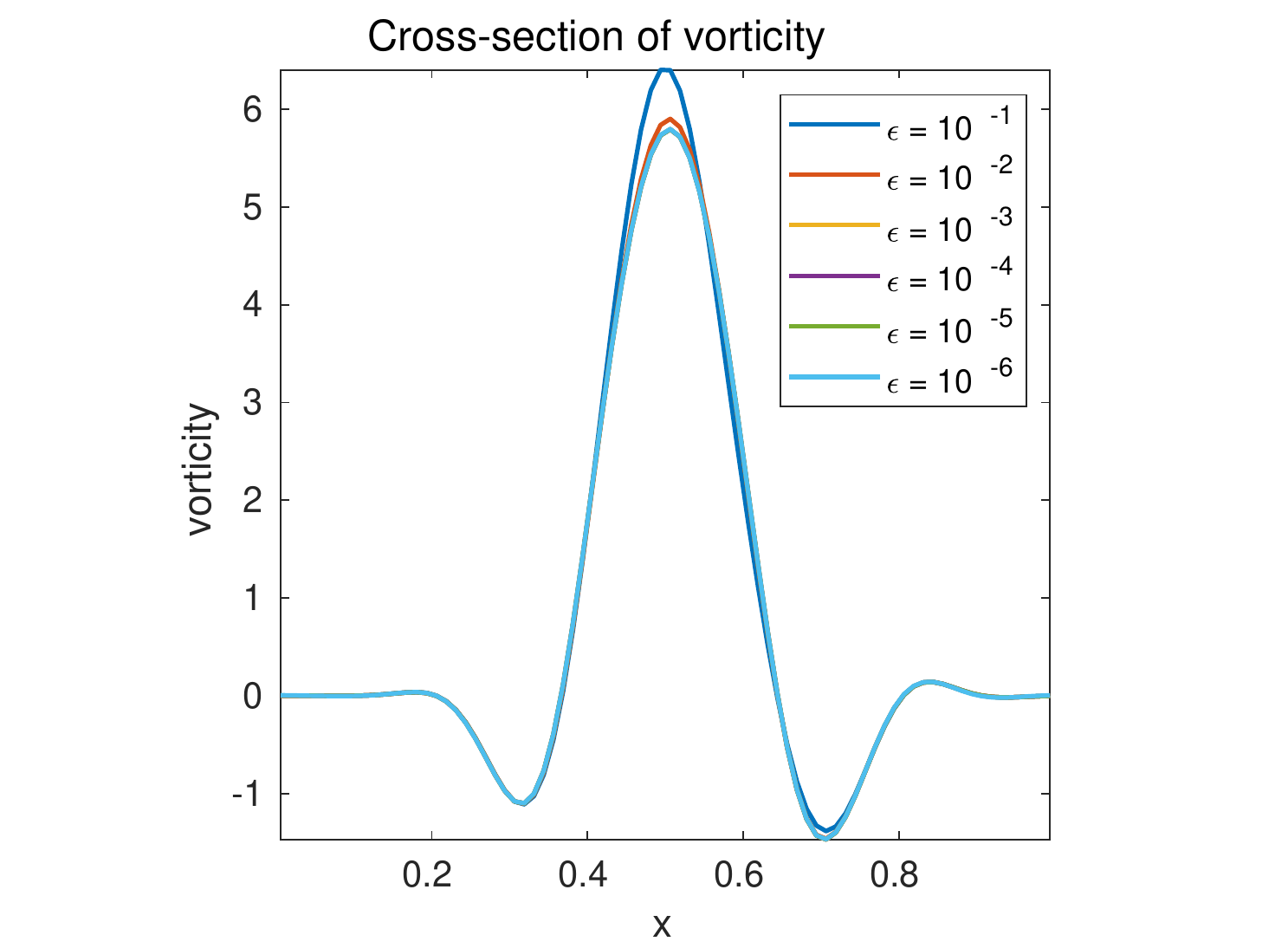}
  \caption{Left: relative kinetic energy from $T=0$ to $T=1.67$ for
    different $\veps$. Right: vorticity at $x_2=0$ and $T=1.67$ for
    different $\veps$.}
  \label{fig:ke_vort}
\end{figure}

\subsection{Asymptotic Order of Convergence}
\label{sec:problem-2}
The aim of this case study is to demonstrate the asymptotic
convergence of the scheme, i.e.\ its ability to converge to the
incompressible solution as $\veps\to0$, and to show that the
convergence rate is two.

We consider the following exact solution:
\begin{equation}
  \label{eq:schneider}
  \begin{aligned}
    u_{1,(0)}(t, x_1, x_2)&=1-2\cos(2\pi(x_1 - t))\sin(2\pi(x_2-t)), \\
    u_{2,(0)}(t, x_1, x_2)&=1+2\sin(2\pi(x_1 - t))\cos(2\pi(x_2-t)), \\
    p_{(2)}(t, x_1, x_2) &=-\cos(4\pi(x_1-t))-\cos(4\pi(x_2-t))    
  \end{aligned}
\end{equation}
of the incompressible system \eqref{eq:incomp} as given in
\cite{schneider}, with $\rho_{(0)}(t, x_1, x_2)=1$. The computational
domain $\Omega=[0,1]\times[0,1]$ is divided into
$10\times10,20\times20,40\times40,80\times80$ mesh cells and the CFL
number is set to $0.45$. The data are initialised using \eqref{eq:schneider}
at $t=0$ and the computations are done up to a final time $T=3$. The
boundaries are periodic everywhere. The EOC computed using the exact
solution \eqref{eq:schneider} as the reference solution is termed as the
asymptotic order of convergence (AOC). Our simulation results show
that the density remains exactly as $\rho\equiv 1$, and we use the
velocity components to measure the AOC. In Table~\ref{tab:aocepsm}, we
present the AOC computed using three very small values of $\veps$,
namely $10^{-6},10^{-5}$ and $10^{-4}$. It is very clear from the data
that the scheme indeed converges to the incompressible solution with
second order accuracy as $\veps\to0$. It is also worth remarking that
the present results are in conformity with the fact that the variant
ARS(2,2,2) used here is GSA \cite{AscherRuuthSpiteri}. We have also
tested non-GSA variants, such as PR(2,2,2), and the results show only
marginal differences.    
\begin{table}[htbp]
  \centering
  \begin{tabular}[htbp]{|c|c|c|c|c|c|c|c|c|}
    \hline
    $N$ & $L^1$ error in $u_1$ & AOC & $L^2$ error in $u_1$ & AOC &
    $L^1$ error in $u_2$ & AOC & $L^2$ error in $u_2$ & AOC \\
    \hline
    10 & 7.5900e-01 & & 9.3643e-01 & & 7.5900e-01 & &
    9.3644e-01 &  \\
    \hline
    20 & 2.7311e-01 & 1.4746 & 3.3303e-01 & 1.4915 & 2.7311e-01
    & 1.4746 & 3.3303e-01 & 1.4915\\
    \hline
    40 & 6.4266e-02 & 2.0874 & 7.5131e-02 & 2.1482 & 6.4266e-02
    & 2.0874 & 7.5131e-02 & 2.1482\\
    \hline
    80 & 1.5283e-02 & 2.0721 & 1.7248e-02 & 2.1230 & 1.5283e-02
    & 2.0721 & 1.7248e-02 & 2.1230\\
    \hline    \hline 
10 & 7.5900e-01 &  & 9.3643e-01 & & 7.5900e-01 & &
 9.3643e-01 &   \\
\hline 
20 & 2.7308e-01 & 1.4748 &	3.3299e-01 & 1.4917 &
                                                          2.7308e-01
    & 1.4748  & 3.3299e-01 & 1.4917 \\
\hline
40 & 6.4183e-02 & 2.0890 & 7.5043e-02 & 2.1497 &
                                                         6.4183e-02
    & 2.0890 & 7.5043e-02 & 2.1497\\ 
\hline
80 & 1.5198e-02 &	2.0783 & 1.7154e-02 & 2.1292 &
                                                           1.5198e-02
    & 2.0783 & 1.7154e-02 & 2.1292 \\ 
\hline \hline 
10 & 7.5900e-01 & &  9.3643e-01  & & 7.5900e-01 & &
                                                              9.3643e-01&
           \\
\hline
20 & 2.7302e-01 & 1.4751 & 3.3292e-01 & 1.4920 &
                                                         2.7302e-01
    & 1.4751 &3.3292e-01 & 1.4920\\
\hline
40 & 6.3787e-02 &	 2.0977 & 7.4589e-02 &	2.1582 &
                                                                 6.3787e-02
    & 2.0977 & 7.4589e-02 & 2.1582 \\
\hline
80 & 1.3905e-02 & 2.1976 & 1.5690e-02 & 2.2490 &
                                                         1.3905e-02
    & 2.1976 & 1.5690e-02 & 2.2490 \\
\hline
\end{tabular}
  \caption{ $L^1$ and $L^2$ errors in $u_1$ and $u_2$, and the AOC for 
    Problem~\ref{sec:problem-2}. Top to bottom: $\veps=10^{-6}$ to
    $\veps=10^{-4} $.} 
  \label{tab:aocepsm}
\end{table}
\subsection{Cylindrical Explosion Problem}
\label{sec:problem-4}
We consider a 2-D cylindrical explosion problem motivated by 
\cite{boscarino_weno_jcp_19, dimarco_loubere_siamjsc_2017} for the
Euler equations \eqref{eq:comp} with a linear relation between the
pressure and density $P(\rho) = \rho$. 

The computations are carried out on the square $\Omega =
[-1,1]^2$. The initial density profile reads
\begin{equation}\label{eq:prob4_density}
\rho(0, x_1, x_2) = \begin{cases}
  1 + \veps^2 , &  \text{if} \ r^2 \leq 1/4,\\
  1, & \text{otherwise}. 
\end{cases}
\end{equation}
In \eqref{eq:prob4_density}, $r = \sqrt{x_1^2 + x_2^2}$ is the distance
of any point $(x_1, x_2)$ from the origin $(0, 0)$. At  time $T = 0$,
the velocity of the fluid is given by 
\begin{equation}
u_1(0, x_1, x_2) = -\frac{\alpha(x_1, x_2)}{\rho(0, x_1, x_2)}
\frac{x_1}{r} ,  \quad
u_2(0, x_1, x_2) = -\frac{\alpha(x_1, x_2)}{\rho(0, x_1, x_2)}
\frac{x_2}{r},
\end{equation}
where the coefficient $\alpha(x_1, x_2)$ is given by $\alpha := \text{max} (0,
1 - r) (1 - e^{-16 r^2})$ and $(u_1, u_2)$ is set to $(0, 0)$, if $r <
10^{-15}$. All the boundaries are assumed to be periodic. The domain is
divided uniformly by a $100\times 100$ mesh. We have used the
JIN(2,2,2) variant to do the time discretisation.  

In order to simulate a compressible regime, we first set
$\veps=1$. The surface plots of the density, and quiver plots of the
velocity field at times $T=0.1,0.24$ and $0.5$ are given in
Figure~\ref{fig:sod}. Clearly, the circular shockwave moving outwards
can be observed in the both the figures, confirming the good
performance of the scheme in the fully compressible case.    
\begin{figure}[htbp]
  \centering
  \includegraphics[height=0.18\textheight]{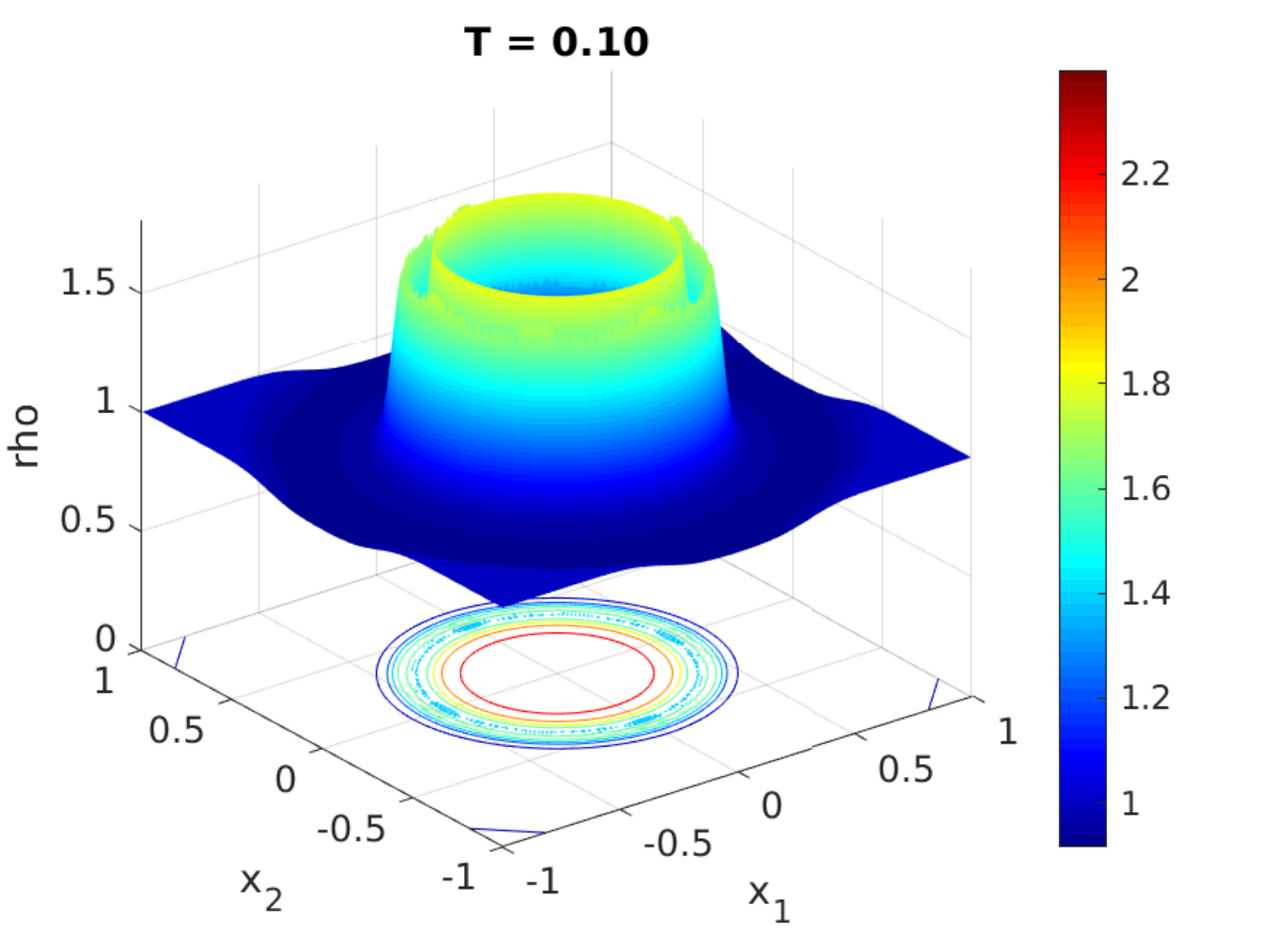}
  \includegraphics[height=0.18\textheight]{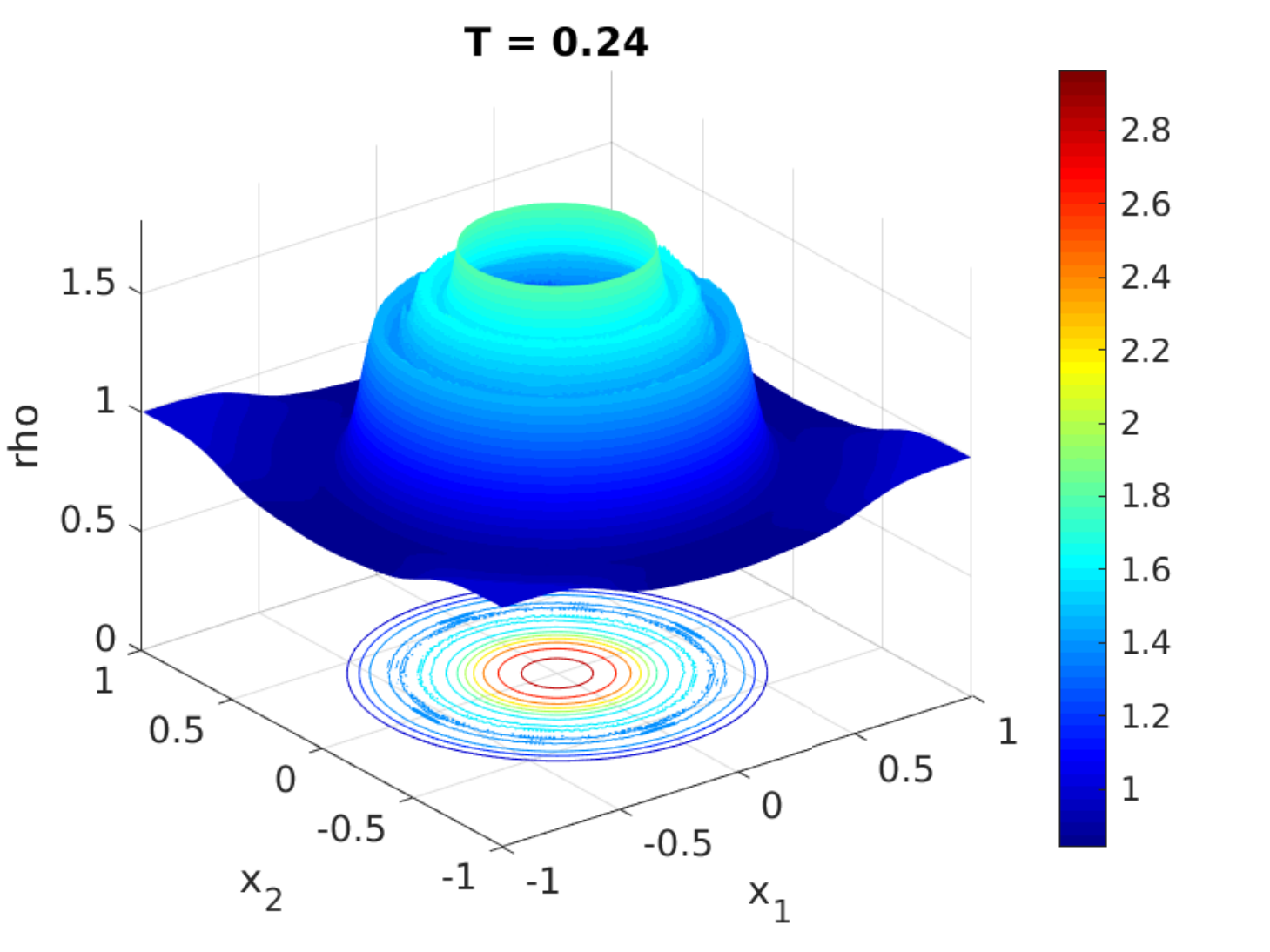} 
  \includegraphics[height=0.18\textheight]{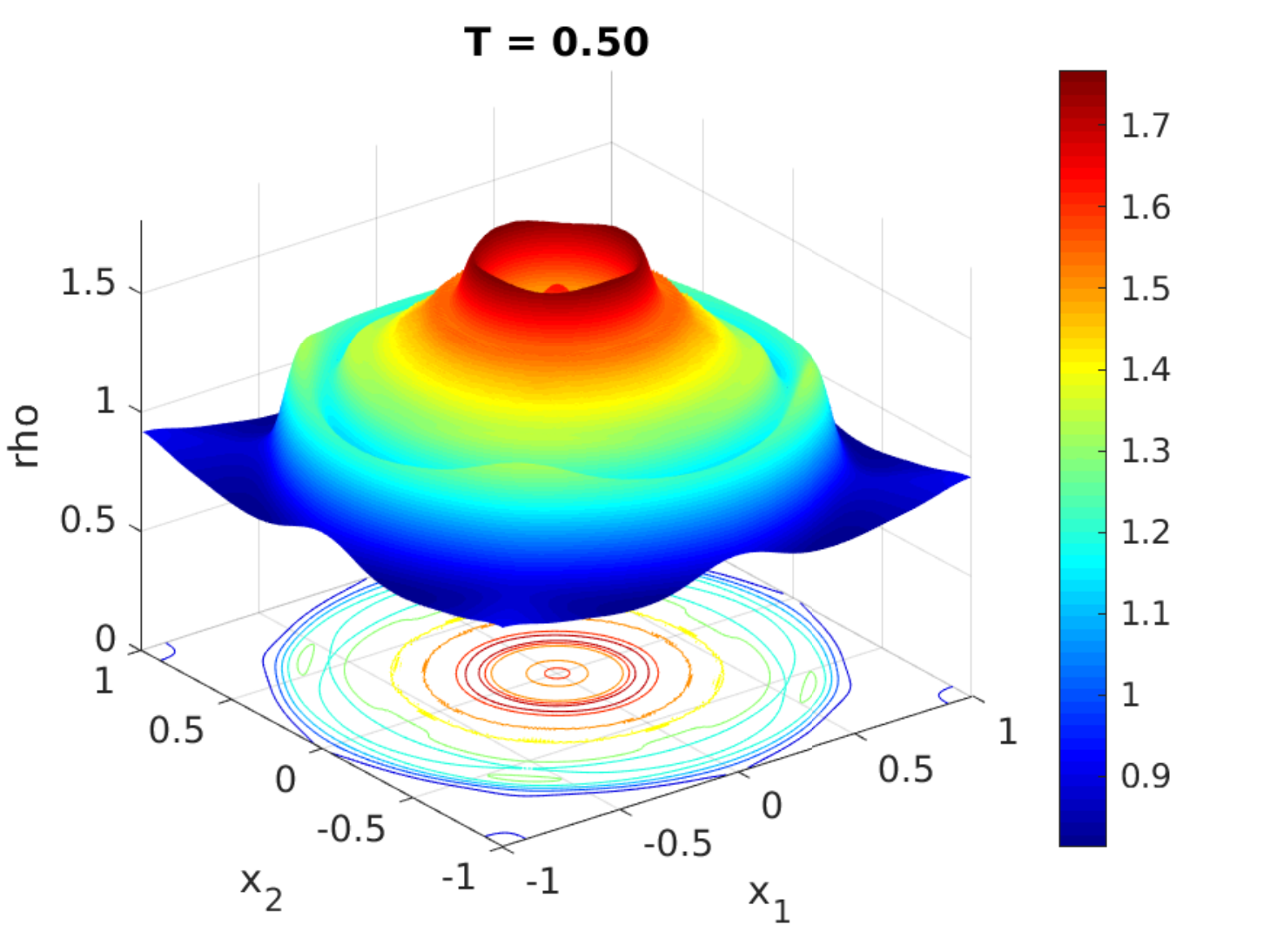}
  \includegraphics[height=0.18\textheight, trim={3cm 0 0cm 0}]{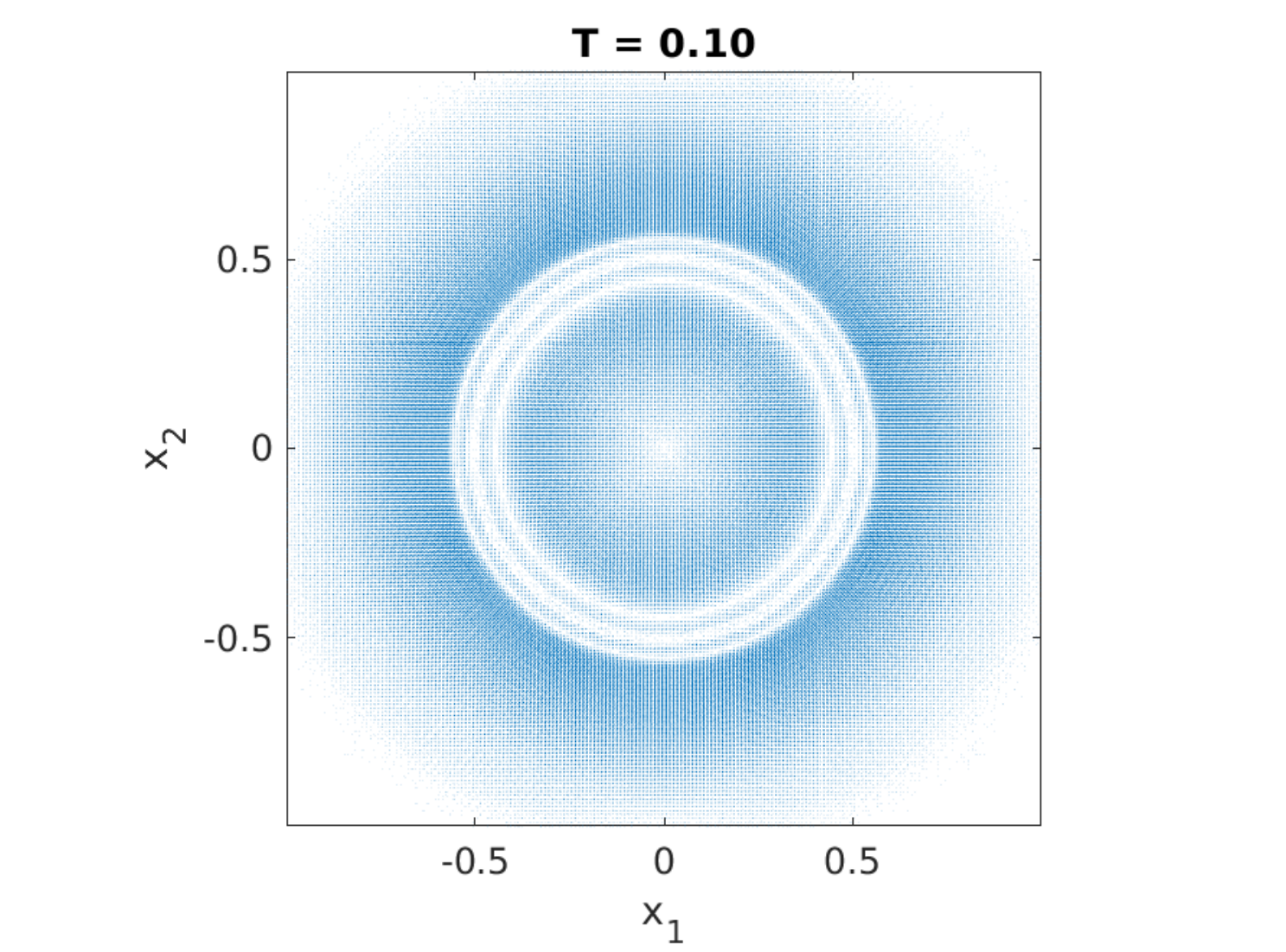}
  \includegraphics[height=0.18\textheight, trim={0.5cm 0 0cm 0}]{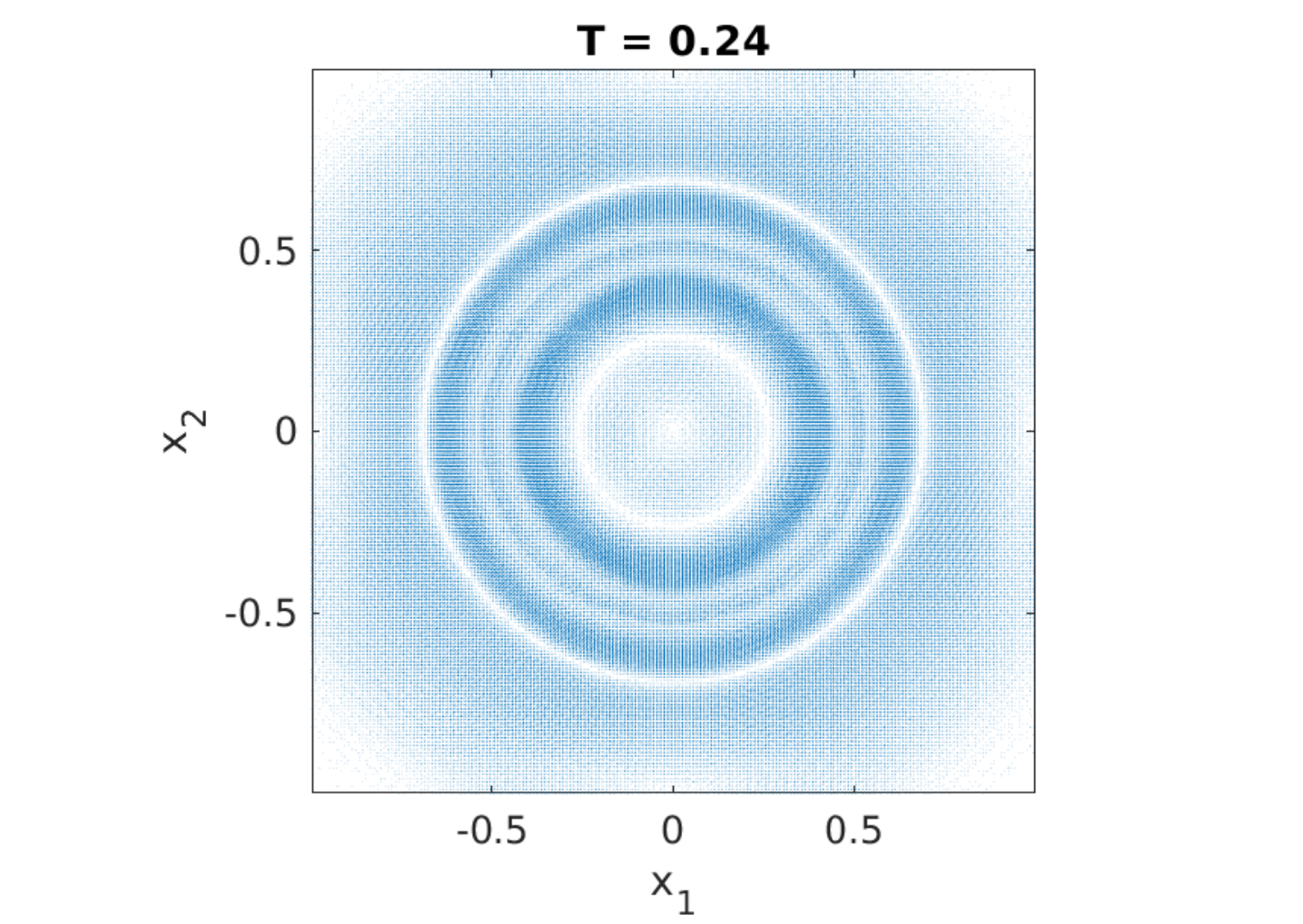} 
  \includegraphics[height=0.18\textheight, trim={0.5cm 0 0cm 0}]{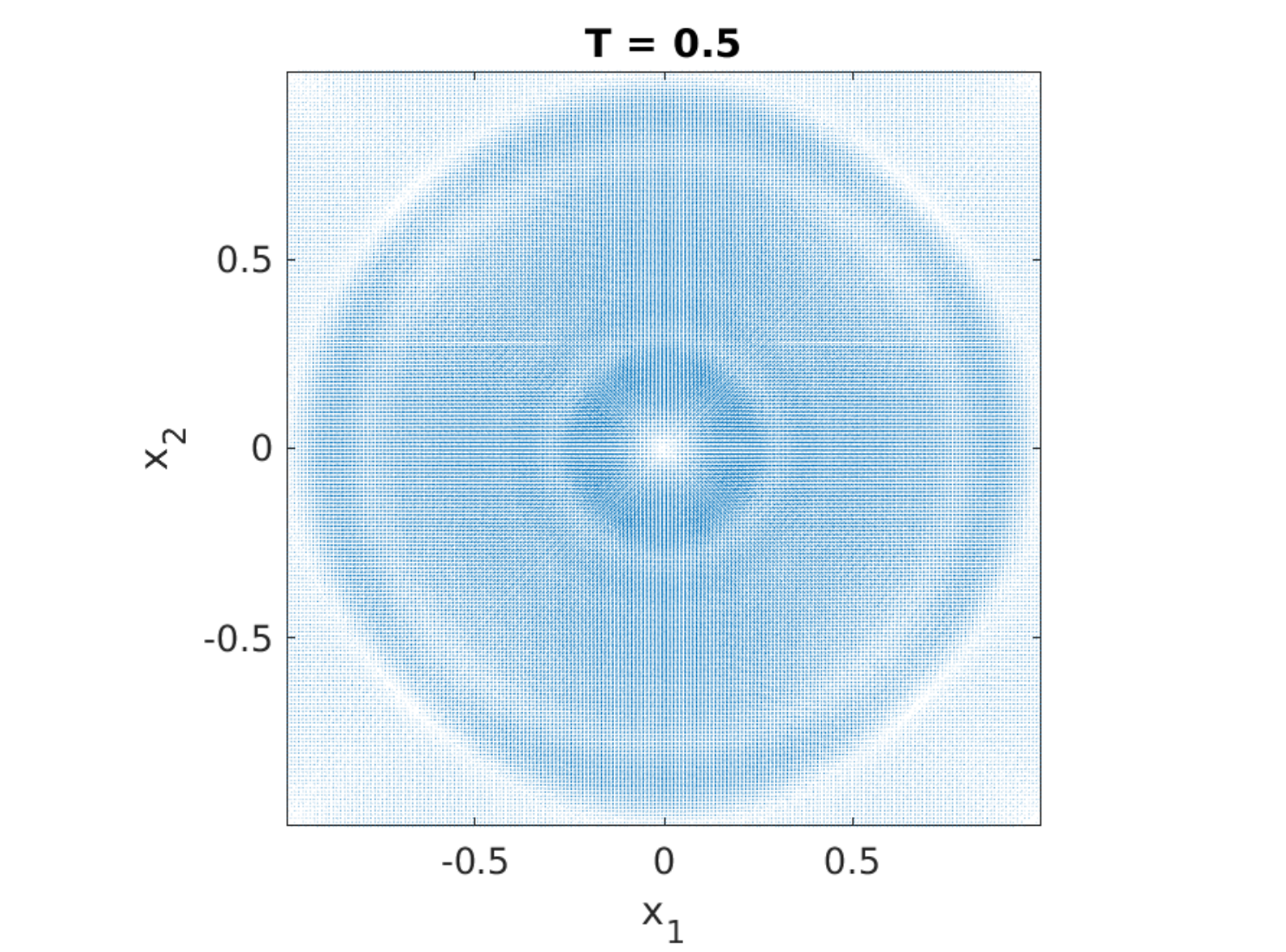} 
  \caption{Surface plots of the density (top panel), and quiver plots
    of the velocity field (bottom panel) for the cylindrical
    explosion problem with $\veps=1$, at times $T= 0.1$ (left), $T =
    0.24$ (middle), and $T = 0.5$ (right).} 
  \label{fig:sod}
\end{figure}

In the same problem, we set $\veps=10^{-3}$ to go to the weakly
compressible or nearly incompressible regime. The surface plots of the
density and divergence of the velocity is given
Figure~\ref{fig:den_dvg_contf}. It can be seen that the density is
equal to one, and the velocity divergence is very small, indicating
the convergence of the solution to the incompressible one.   
\begin{figure}[htbp]
  \centering
  \includegraphics[height=0.25\textheight]{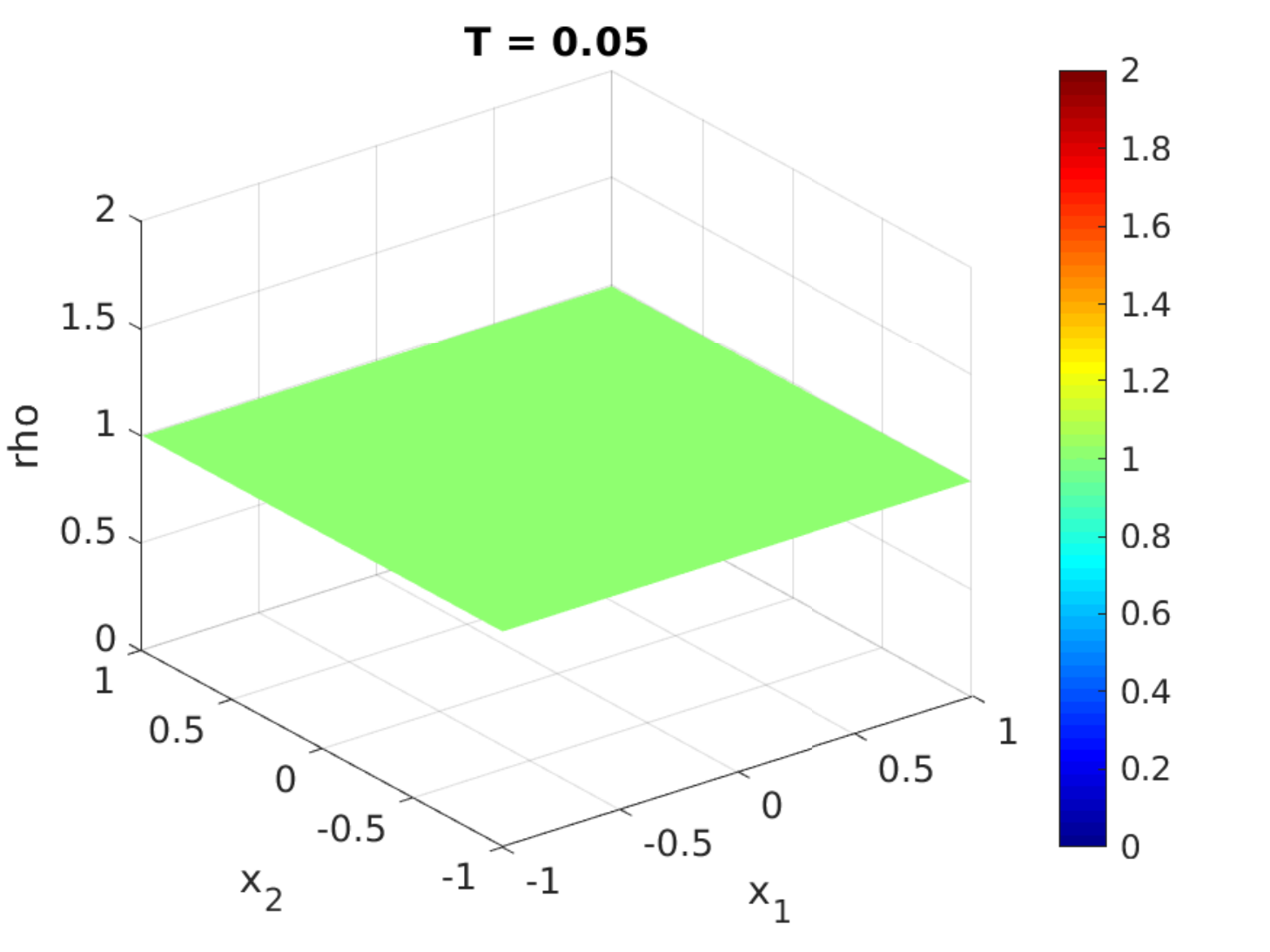}
  \includegraphics[height=0.25\textheight]{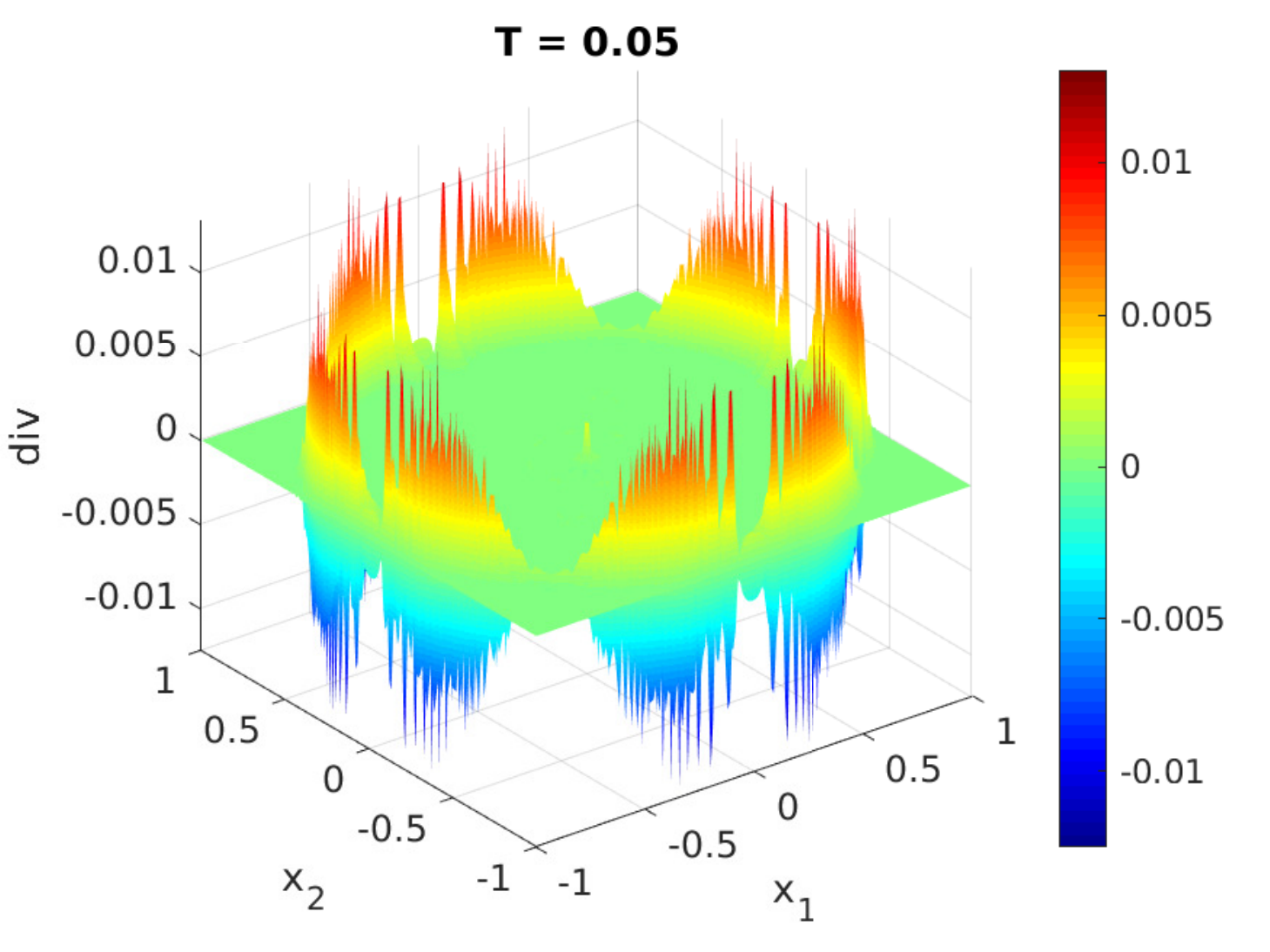} 
  \caption{Surface plot of the density (left) and divergence of the
    velocity (right) for the cylindrical explosion problem with
    $\veps=10^{-3}$ at time $T= 0.01$.} 
  \label{fig:den_dvg_contf}
\end{figure}

\section{Concluding Remarks}
\label{sec:sec8}
We have presented a class of second order accurate IMEX-RK finite
volume schemes for the compressible Euler equations in the zero Mach
number limit. Guided by the results of asymptotic analysis, the
nonlinear fluxes in the Euler equations are split into stiff and
non-stiff parts in which stiff terms correspond to fast acoustic waves
and the non-stiff terms to slow advection waves. In the time
discretisation, the stiff terms are treated implicitly, and the
non-stiff terms explicitly. The resulting time semi-discrete scheme is
shown to be AP and AA. In order to derive a space-time fully discrete
scheme, a Rusanov-type central flux used for the non-stiff terms and
central differences for stiff terms. The fully-discrete scheme is also
shown to be AP and AA. A sufficient for $L^2$-stability, involving only
the RK coefficients, overcoming the stiffness due to CFL restrictions
is proposed for the time semi-discrete as well as space-time
fully-discrete schemes. The numerical experiments confirm that the
scheme achieves uniform second order convergence with respect to the
Mach number, its dissipation is independent of the Mach number, and it 
converges to the incompressible solution with second order convergence
as the Mach number goes to zero.

\section*{Acknowledgements}
\label{sec:acknowledgements}

The authors thank the anonymous referees whose comments have lead to a
significant improvement of the manuscript.

\section*{Appendix}
\label{sec:appendix}

First, let us introduce the following shorthands.
\begin{equation}\label{eq:def_b_4}
  \begin{aligned}
    b^{(2)}_1 &= \tilde{\omega}_2 \tilde{c}_{2} - \frac{1}{2},  \quad
    b^{(2)}_2= \tilde{\omega}_1 c_{1}  + \tilde{\omega}_2
    c_{2} + \omega_2 \tilde{c}_{2} - 1, \quad
    b^{(2)}_3=\omega_1c_{1}+\omega_{2} c_{2}-\frac{1}{2}, \\
    b^{(3)}_1&= \frac{1}{2} - \left\{\tilde{\omega}_2 \tilde{c}_{2} \sum_{i=1}^{2}a_{i,i} +
    \tilde{\omega}_3 \sum_{j=1}^{3} ( \tilde{a}_{3,j} c_{j} + a_{3,j}
    \tilde{c}_{j}) + \omega_3 \tilde{c}_{2} \tilde{a}_{3,2} \right\}, \quad 
    b^{(3)}_2  = \frac{1}{6} - \tilde{\omega}_3 \tilde{c}_{2} \tilde{a}_{3,2} \\
    b^{(3)}_3 & = \frac{1}{2} - \left\{ \sum_{i = 1}^{3
      }\tilde{\omega}_i\sum_{j = 1}^{i } a_{i,j} c_{j} 
    +  \omega_2 \tilde{c}_{2}\sum_{i=1}^{2}a_{i,i}+ \omega_3\sum_{j =
      1}^{3} ( \tilde{a}_{3,j} c_{j} + a_{3,j} \tilde{c}_{j})
  \right\}, \quad 
    b^{(3)}_4 =  \frac{1}{6} - \sum_{i,j=1}^{3} \omega_i a_{i,j} c_{j}  \\
    b^{(4)}_1 &= \tilde{\omega}_2
    \tilde{c}_{2}\left(a_{1,1}\sum_{i=1}^2a_{i,i} +
      a_{2,2}^2\right)+\tilde{\omega}_3 \left(\sum_{j= 1}^2 
    \tilde{a}_{3,i} \sum_{i=1}^j c_{i} a_{j,i} + a_{3,2} \tilde{c}_{2}
    \sum_{i=1}^2a_{i,i} + a_{3,3} \sum_{i=1}^2 \tilde{a}_{3,i} c_{i} +
    \sum_{j=2}^3 a_{3,j} \tilde{c}_{j}\right)\\
  & \quad + \omega_{3} \tilde{a}_{3,2} \tilde{c}_{2} \sum_{i=1}^3
  a_{i,i} - \frac{1}{4} , \quad
  b^{(4)}_2 = \tilde{\omega}_{3} \tilde{a}_{3,2}\tilde{a}_{2,1}
    \sum_{i=1}^{3}a_{i,i} - \frac{1}{6}\\ 
    b^{(4)}_3 &= \sum_{k=1}^3 \tilde{\omega}_{k} \sum_{j = 1}^k
    a_{k,j} \sum_{i = 1}^j a_{j,i} c_{i} + \sum_{k=2}^3\omega_{k}
    \sum_{j=1}^{k-1} \tilde{a}_{k,j} \sum_{i=1}^j  a_{j,i} c_i +
    \sum_{k=2}^3 \omega_k \tilde{a}_{2,1} a_{k,2} \sum_{i=1}^2a_{i,i} \\
    &+ \omega_{3} a_{3,3} \sum_{j=1}^2 (\tilde{a}_{3,j} c_j + a_{3,j+1}
    \tilde{c}_{j+1}) - \frac{1}{6},  \quad
    b^{(4)}_4=  \sum_{k = 1}^3\omega_{k} \sum_{j=1}^k a_{k,j} \sum_{i
      = 1}^j a_{j,i} c_i - \frac{1}{24} 
  \end{aligned}
\end{equation}
With the above notations, the entries in the matrices
$\mcal{B}^{(2)},B^{(3)}$ and $B^{(4)}$ read
\begin{equation*}
  \begin{aligned}
    \mcal{B}_{1,1}^{(2)}&=\Dlt \left\{ b^{(2)}_1 ({\bu} \cdot \nabla
      )^2+b^{(2)}_3\frac{{\ba}^2}{\veps^2} \Delta
    \right\}+\frac{1}{2}\Delta x_k\abs{{{\baru}}_k}\partial_{x_k}^2,
    \quad 
    \mcal{B}_{1,2}^{(2)}=\Dlt {\br}b^{(2)}_2({\bu}\cdot\nabla
    ) \nabla \cdot,  \\
    \mcal{B}_{2,1}^{(2)}&=\Dlt
    \frac{{\ba}^2}{{\br}\veps^2}b^{(2)}_2({\bu} \cdot \nabla)
    \nabla, \quad
    \mcal{B}_{2,2}^{(2)}=\Dlt \left\{b^{(2)}_1({\bu} \cdot
      \nabla)^2+
      b^{(2)}_3\frac{{\ba}^2}{\veps^2}\nabla (\nabla \cdot) \right\}+\frac{1}{2}\Delta
    x_k\abs{{{\baru}}_k}\partial_{x_k}^2\mbb{I}_2.
\end{aligned}
\end{equation*}
\begin{equation*}
  \label{eq:ieu_B3_defn_TD}
  \begin{aligned}
    B^{(3)}_{1,1}&=b^{(3)}_2 ({\bu} \cdot \nabla)^3 +
    b^{(3)}_3\frac{{\ba}^2}{\veps^2}({\bu} \cdot \nabla) \Delta, 
    \quad
    B^{(3)}_{1,2}=b^{(3)}_1 \br ({\bu} \cdot \nabla)^2 \nabla \cdot + b^{(3)}_4
    \frac{{\ba}^2 {\br}}{\veps^2} \Delta \nabla \cdot, \\
    B^{(3)}_{2,1}&=b^{(3)}_1 \frac{{\ba}^2}{\veps^2} ({\bu} \cdot
    \nabla)^2 \nabla + b^{(3)}_4 \frac{{\ba}^4 }{{\br}\veps^4}
    \Delta \nabla,  \quad
    B^{(3)}_{2,2}=b^{(3)}_2({\bu} \cdot \nabla)^3 \mbb{I}_2+b^{(3)}_3
    \frac{{\ba}^2}{\veps^2}({\bu} \cdot \nabla) \nabla \nabla
    \cdot.
  \end{aligned}
\end{equation*}
\begin{equation*}
  \label{eq:ieu_B4_defn_TD}
  \begin{aligned}
    B^{(4)}_{1,1}&= b^{(4)}_1 \frac{{\ba}^2}{\veps^2}({\bu} \cdot \nabla)^2 \Delta +
      b^{(4)}_4\frac{{\ba}^4}{\veps^4} \Delta^2 - \frac{1}{24}
      ({\bu} \cdot \nabla)^4,  \quad
      B^{(4)}_{1,2}=b^{(4)}_2{\br} ({\bu} \cdot \nabla)^3 \nabla
      \cdot +b^{(4)}_3 \frac{{\ba}^2 {\br}}{\veps^2} ({\bu} \cdot \nabla)\Delta
      \nabla \cdot, \\
      B^{(4)}_{2,1}&= b^{(4)}_2 \frac{{\ba}^2}{{\br}
        \veps^2}({\bu} \cdot \nabla)^3 \nabla +
      b^{(4)}_3\frac{{\ba}^4 }{{\br}\veps^4}({\bu} \cdot
      \nabla)\Delta \nabla,  \
      B^{(4)}_{2,2}=b^{(4)}_1 \frac{{\ba}^2}{\veps^2}({\bu} \cdot
      \nabla)^2 \nabla (\nabla \cdot) +
      b^{(4)}_4\frac{{\ba}^4}{\veps^4} \nabla \Delta \nabla \cdot -
      \frac{1}{24} ({\bu} \cdot \nabla)^4 \mbb{I}_2.
  \end{aligned}
\end{equation*}

\begin{figure}[htbp]
  \centering
  \begin{tabular}{c|c c}
    0   & 0	& 0	\\
    1   & 1      & 0  \\
    \hline 
        & 1	 & 0	
  \end{tabular}
  \hspace{10pt}
  \begin{tabular}{c|c c}
    0   & 0	& 0	\\
    1   & 0      & 1  \\
    \hline 
        & 0	 & 1	
  \end{tabular}
  \hspace{30pt}
  \begin{tabular}{c|c c}
    0   & 0	& 0	\\
    1   & 1      & 0  \\
    \hline 
        & $\frac{1}{2}$	 & $\frac{1}{2}$	
  \end{tabular}
  \hspace{10pt}
  \begin{tabular}{c|c c}
    -1   & -1	& 0	\\
    2   & 1      & 1  \\
    \hline 
         & $\frac{1}{2}$	 & $\frac{1}{2}$
  \end{tabular}
  \hspace{30pt}
   \begin{tabular}{c|c c}
    0   & 0	& 0	\\
    1   & 1      & 0  \\
    \hline 
        & $\frac{1}{2}$	 & $\frac{1}{2}$	
  \end{tabular}
  \hspace{10pt}
  \begin{tabular}{c|c c}
    $1-\gamma_p$   & $1-\gamma_p$	& 0	\\
    $\gamma_p$  & $\gamma_p - \delta_p$      & $\delta_p$  \\
    \hline 
         & $\frac{1}{2}$	 & $\frac{1}{2}$
  \end{tabular}
  
  \vspace*{5mm}
   \begin{tabular}{c|c c c}
     0		& 0			& 0			& 0	\\
     $\gamma_a$	& $\gamma_a$		& 0 			& 0	\\
     $1$	& $\delta_a$		& $1 - \delta_a$		& 0	\\
     \hline 
                & $\delta_a$	& $1 - \delta_a$	& $0$
   \end{tabular}
   \hspace{10pt}
   \begin{tabular}{c|c c c}
     $0$	& $0$		& 0			& 0		\\
     $\gamma_a$	& $0$		& $\gamma_a$ 		& 0		\\
     $1$	& $0$		& $1 - \gamma_a$		& $\gamma_a$	\\
     \hline 
		& $0$		& $1 - \gamma_a$		& $\gamma_a$
   \end{tabular}
   \hspace{30pt}
   \begin{tabular}{c|c c c}
     0		& 0			& 0			& 0	\\
     $\frac{1}{2}$	& $\frac{1}{2}$		& 0 			& 0	\\
     $1$	& $1$		& $1$		& 0	\\
     \hline 
                & $0$	& $1$	& $0$
   \end{tabular}
   \hspace{10pt}
   \begin{tabular}{c|c c c}
     $0$	& $0$		& 0			& 0		\\
     $\frac{1}{2}$	& $0$		& $\frac{1}{2}$ 		& 0		\\
     $1$	& $\frac{1}{2}$		& $0$		& $\frac{1}{2}$	\\
     \hline 
                & $\frac{1}{2}$		& $0$		& $\frac{1}{2}$
   \end{tabular}
  \caption{Double Butcher tableaux of IMEX-RK schemes. Top left: Euler
  (1,1,1), top middle: JIN(2,2,2), top right: PR(2,2,2), where
  $\delta_p=1-(1/2\gamma_p)$, bottom left: ARS(2,2,2), where
  $\gamma_a=1-(\sqrt{2}/2), \ \delta_a=1-(1/2\gamma_a)$, and
  bottom right: CN(2,2,2).}
  \label{fig:double_butcher}
\end{figure}

\bibliography{references}
\bibliographystyle{abbrv}

\end{document}